\documentclass[11pt]{article}
\usepackage{amsfonts}
\usepackage{amssymb}
\usepackage{amsthm}
\usepackage{amsmath}
\usepackage{graphicx}
\usepackage{empheq}
\usepackage{indentfirst}
\usepackage{cite}
\usepackage{mathrsfs}
\usepackage{graphics}
\usepackage{enumerate}
\usepackage{color}

 \newtheorem{theorem}{Theorem}[section]
 \newtheorem{corollary}{Corollary}[section]
 \newtheorem{lemma}{Lemma}[section]
 \newtheorem{proposition}{Proposition}[section]

 \newtheorem{remark}{Remark}[section]
 \numberwithin{equation}{section}

\allowdisplaybreaks[4]

\newcommand{\nn}{\nonumber}
\newcommand{\io}{\int_\Omega}
\newcommand{\R}{\mathbb{R}}

\newcommand{\beq}{\begin{equation}}
\newcommand{\eeq}{\end{equation}}
 \def\non{\nonumber }
\def\bea{\begin{eqnarray}}
\def\eea{\end{eqnarray}}

\begin{document}
\title{Comparison Methods for a Keller--Segel-type Model of Pattern Formations with Density-suppressed Motilities}
\author{Kentarou Fujie\thanks{Research Alliance Center for Mathematical Sciences, Tohuku University, Sendai 980-8578, Miyagi, Japan, \textsl{fujie@tohoku.ac.jp}},
	\ Jie Jiang\thanks{Wuhan Institute of Physics and Mathematics, Chinese Academy of Sciences,
		Wuhan 430071, HuBei Province, P.R. China,
		\textsl{jiang@wipm.ac.cn}.}}

\date{\today}

\maketitle

\begin{abstract} 
This paper is concerned with global well-posedness to the following fully parabolic kinetic system
\begin{equation}
\begin{cases}\label{chemo0}
u_t=\Delta (\gamma (v)u)\\
v_t-\Delta v+v=u
\end{cases}
\end{equation}
in a smooth bounded domain $\Omega\subset\mathbb{R}^n$, $n\geq1$ with no-flux boundary conditions. This model was recently proposed in \cite{PRL12,Sciencs11} to describe the process of stripe pattern formations via the so-called self-trapping mechanism. The system features a signal-dependent motility function $\gamma(\cdot)$ which is decreasing in $v$ and will vanish as $v$ tends to infinity. 

The major difficulty in analysis comes from the possible degeneracy as $v\nearrow+\infty.$ In this work we develop a new comparison method different from the conventional energy method in literature which reveals a striking fact that there is no finite-time degenercay in this system. More precisely, we use comparison principles for elliptic and parabolic equations to prove that degeneracy cannot take place in finite time in any spatial dimensions for all smooth motility functions satisfying $\gamma(s)>0$, $\gamma'(s)\leq0$ when $s\geq0$ and $\lim\limits_{s\rightarrow+\infty}\gamma(s)=0.$ Then we investigate global existence of classical solutions to \eqref{chemo0}  when $n\leq3$ and discuss the uniform-in-time boundedness under  certain growth conditions on $1/\gamma.$

In particular, we consider system \eqref{chemo0} with $\gamma(v)=e^{-v}$, which shares the same set of equilibria as well as the Lyapunov functional as the classical Keller--Segel model. In the two-dimensional setting, we observe a critical-mass phenomenon which is distinct from the well-known fact for the classical Keller--Segel model. We prove that classical solution always exists globally which is uniformly-in-time bounded with arbitrary initial data of sub-critical mass. On the contrary, with certain initial data of super-critical mass, the solution will become unbounded at time infinity which differs from the finite-time blowup behavior of the Keller--Segel model. 
	
\end{abstract}
\section{Introduction}
Recently,  Fu et al. \cite{PRL12}  proposed a fully parabolic kinetic system to model the process of stripe pattern formation through the so-called self-trapping mechanism. Denote the density of cells and the concentration of signals by $u(x,t)$ and $v(x,t)$, respectively. The resulting system reads
\begin{equation}
\begin{cases}\label{chemo0a}
u_t=\Delta (\gamma (v)u)+\mu u(1-u)\\
 \varepsilon v_t-\Delta v+v=u,
\end{cases}
\end{equation}
where  $\mu,\varepsilon\geq0$ are given constants. Here, $\gamma(\cdot)$ is a signal-dependent motility function decreasing in $v$ which characterizes the repressesive effect of signal concentration on cell motility. As experimentally observed in \cite{Sciencs11,PRL12}, this model correctly captures the dynamics at the propagating front where new stripes are formed.

Note  that $\Delta (\gamma(v)u)=\nabla \cdot(\gamma(v)\nabla u)+\nabla \cdot(u\gamma'(v)\nabla v)$.  The first equation of \eqref{chemo0a} has the following variant form
\begin{equation}\label{decom}
	u_t-\nabla \cdot(\gamma(v)\nabla u)=\nabla\cdot(u\gamma'(v)\nabla v)+\mu u(1-u).
\end{equation}
Since $\gamma'\leq 0$, system \eqref{chemo0a} can be regarded as a chemotaxis model of  Keller--Segel type involving signal-dependent diffusion rates and chemo-sensitivities.

Apparantly, the dependence of diffusion rate on $v$  leads to possible degeneracy as $v$ becomes unbounded.  Theoretical results concerning global solvability or existence of blowup are rather limited in the literature. In \cite{TaoWin17}, Tao and Winkler considered the initial-boundary value problem of \eqref{chemo0a} with $\mu=0$ and $\varepsilon=1$. By assuming uniform lower and upper bounds of $\gamma$ and $\gamma'$, they obtained global existence of uniformly-in-time bounded classical solutions in two dimensions and the existence of global weak solutions in higher dimensions. Global existence of classical solutions in the three-dimensional case was also examined under certain smallness assumptions on the initial data.

If $\gamma(v)$ vanishes as $v$ tends to infinity, then degeneracy becomes a serious issue in analysis. Therefore, the key problem lies in deriving an upper bound for $v$. One classical way in literature is to increase the $L^p-$integrability of $u$ since the $L^\infty(0,T;L^p(\Omega))$ boundedness of $u$ will yield to an upper bound for $v$ via the second equation with any $p>\frac{n}{2}$. Along with this idea,  Yoon and Kim \cite{YK17} studied \eqref{chemo0a} with a specific motility function $\gamma(v)=c_0 v^{-k}$, $\varepsilon=1$  and $\mu=0.$ By introducing approximating step functions of the motility, they obtained global existence of classical solution which is uniformly-in-time bounded for all $k>0$ under a smallness assumption on $c_0>0.$

On the other hand, the presence of logistic growth terms also helps to achieve higher $L^p-$integrability of $u$. In \cite{JKW18}, the degeneracy issue was tackled with the aid of the logistic source where global existence of uniformly-in-time bounded classical solutions  was proved  with any $\mu>0$ when $n=2$ and $\varepsilon=1$. However, a crucial assumption made in their work is that $\lim\limits_{s\rightarrow+\infty}\frac{\gamma'(s)}{\gamma(s)}$ exists which  excludes fast decay motilities like $e^{-v^2}$ or $e^{-e^v}$. More recently in \cite{WW2019},  making use of the approach developed by Winkler \cite{Win10} in the study of Keller--Segel model with logistic sources together with the approximating idea in \cite{YK17}, global existence of uniformly-in-time bounded classical solutions was shown when $n\geq3$ with large $\mu>0$ under an assumption of uniform boundedness of $|\gamma'(\cdot)|$ on $[0,\infty)$.

From a mathematical point of view, the problem becomes even challenging when $\mu=0$. To the best of our knowledge, global existence without any smallness assumption or logistic sources was only achieved in the simplified parabolic-elliptic case, i.e., $\varepsilon=0.$ With a   specific motility $\gamma(v)=v^{-k}$, global existence of classical solution with a uniform-in-time bound was established by delicate energy estimates in \cite{Anh19} when $n\leq 2$ for any $k>0$ or $n\geq3$ for $k<\frac{2}{n-2}$. 

In all work mentioned above, the upper bound of $v$ was established via deriving the $L^p-$integrability of $u$ with $p>\frac{n}{2}$ by energy method. Most calculations were carried out relied on the more familiar variant form \eqref{decom}. However, it should be noted that the decomposition in \eqref{decom} also breaks the original delicate structure and omits some significant information.  Recently in \cite{FJ19}, we considered the simplified parabolic-elliptic version of system \eqref{chemo0a} with general motility functions that satisfy
\begin{equation}\label{gamma0}
\mathrm{(A0)}:\gamma(v)\in C^3[0,+\infty),\;\gamma(v)>0,\;\;\gamma'(v)\leq0\;\;\text{on}\;(0,+\infty).
\end{equation}
Keeping the integrity of $\Delta(\gamma(v)u)$ in the first equation, we made a subtle observation of the nonlinear coupling structure. A new method based on comparison principle for elliptic equations was introduced to derive directly the point-wise upper bounds of $v$. Thus, finite-time degeneracy cannot take place for all $n\geq1$. Then we showed that classical solution always exists globally in dimension two under the assumption $\mathrm{(A0)}$ with any $\mu\geq0$. Moreover, the global solution was proven to be uniformly-in-time bounded if either $\mu>0$ or $1/\gamma$ satisfies certain polynomial growth condition.  More importantly, occurrence of exploding solutions was examined for the first time for this signal-dependent model. In the case $\gamma(v)=e^{-v}$ and $\mu=0$, a novel critical-mass phenomenon in the  two-dimensional setting was observed  that with any sub-critical mass, the global solution is uniformly-in-time bounded while with certain super-critical mass,  the global solution  will blow up at time infinity.

In this paper, we study the initial-boundary value problem for the original doubly parabolic degenerate system:
\begin{equation}
\begin{cases}\label{chemo1}
u_t=\Delta (\gamma (v)u)&x\in\Omega,\;t>0\\
v_t-\Delta v+v=u&x\in\Omega,\;t>0\\
\partial_\nu u=\partial_\nu v=0,\qquad &x\in\partial\Omega,\;t>0\\
u(x,0)=u_0(x),\;\;v(x,0)=v_0(x),\qquad & x\in\Omega,
\end{cases}
\end{equation}where $\Omega\subset\mathbb{R}^n$ with $n\geq1$  is a smooth bounded domain.


Our motivation comes from the typical choice  $\gamma(v)=e^{-v}$ in \eqref{chemo1}. Recall that the first equation of \eqref{chemo1} has a variant form \eqref{decom}, which allows us to regard system \eqref{chemo1} as a Keller--Segel system with signal-dependent diffusion rates and chemo-sensitivities. Under the circumstance, our system reads
\begin{equation}\label{chemo2a}
\begin{cases}
u_t=\Delta (ue^{-v})=\nabla \cdot(e^{-v}(\nabla u-u\nabla v)),&x\in\Omega,\;t>0\\
v_t-\Delta v+v=u,&x\in\Omega,\;t>0,
\end{cases}
\end{equation}
which has certain important features in common with the classical/minimal fully parabolic Keller--Segel system:
\begin{equation}\label{ks}
\begin{cases}
u_t=\nabla\cdot(\nabla u-u\nabla v)\\
v_t-\Delta v+v=u\\
\partial_\nu u=\partial_\nu v=0.
\end{cases}
\end{equation}
Indeed, beyond the formal resemblance, they share the same set of equilibria which consists of solutions to the following stationary problem:
\begin{equation}
\begin{cases}\label{steady}
-\Delta v+v=\Lambda e^{v}/\int_\Omega e^{v}\,dx\;\;\text{in}\;\Omega\\
u=\Lambda e^{v}/\int_\Omega e^{v}\,dx\;\;\text{in}\;\Omega\\
\partial_\nu v=0\;\;\text{on}\;\partial\Omega
\end{cases}
\end{equation}with $\Lambda=\|u_0\|_{L^1(\Omega)}>0$.
In addition, they have the same Lyapunov functional. Define the Lyapunov functional by
\begin{equation*}
\mathcal{F}(u,v)=\int_\Omega \left(u\log u+\frac12|\nabla v|^2+\frac12 v^2-uv\right)dx.
\end{equation*}Then for any smooth solution $(u,v)$ of  classical Keller--Segel system \eqref{ks}, there holds
\begin{equation*}
\frac{d}{dt}\mathcal{F}(u,v)(t)+\int_\Omega u\left|\nabla \log u-\nabla v\right|^2dx+\|v_t\|^2_{L^2(\Omega)}=0,
\end{equation*} while for our system \eqref{chemo2a}, there  holds
\begin{equation}\label{Lyapunov1}
\frac{d}{dt}\mathcal{F}(u,v)(t)+\int_\Omega ue^{-v}\left|\nabla \log u-\nabla v\right|^2dx+\|v_t\|^2_{L^2(\Omega)}=0,
\end{equation} where an extra weighted function $e^{-v}$ appears in the second dissipation term.

It is well-known that the classical solutions of the Keller--Segel system \eqref{ks} may blow up when $n\geq2$, i.e., there exists $T_\mathrm{max}\in(0,+\infty]$ such that 
\begin{equation*}
	\lim\limits_{t\nearrow T_{\mathrm{max}}}(\|u(\cdot,t)\|_{L^\infty(\Omega)}+\|v(\cdot,t)\|_{L^\infty(\Omega)})=+\infty.
\end{equation*}
In particular, when $n=2$, the classical Keller--Segel system \eqref{ks} has a critical-mass phenomenon. More precisely, there is a threshold number $\Lambda_c>0$ such that if the conserved total mass is less than $\Lambda_c$, then  global classical solution exists and remains bounded for all time \cite{Nagai97}; otherwise, it may blow up in finite or infinite time \cite{HW01,ssMAA2001}. Recently, a finite-time blowup solution was constructed in \cite{mizoguchi_winkler} and to our knowledge, infinite-time blowup has not been examined yet for the classical fully parabolic Keller--Segel system \eqref{ks} (see \cite{BCM10,GM18} for infinite-time blowup in Cauchy problem of the simplified parabolic--elliptic Keller--Segel system and see \cite{CS,TaoWincritical,Laurencot} for infinite-time blowup in initial-boundary value problem in different kinds of chemotaxis models). In higher dimensions, on the one hand global calssical solution exists with sufficiently small initial data in the scaling-invariant spaces \cite{Cao,Win10} while on the other hand, finite-time blowup was oberved for initial data with arbitrarily small mass \cite{Win13}.

In view of the same steady states of the above two systems  \eqref{chemo2a} and \eqref{ks} as well as the slight difference in dissipations during the evolutionary process, the main purpose of the present paper is to figure out whether their solutions have similar dynamical behavior. 

Now, we summarize the main results of problem \eqref{chemo1} as follow. 
 \begin{enumerate}[(I)]
 	\item  When $n=2$, we prove global existence of classical solution for all motility functions that have a vanishing limit, i.e., $\lim\limits_{s\rightarrow+\infty}\gamma(s)=0$ and satisfy $(\mathrm{A0})$. Moreover,  uniform-in-time boundedness is obtained provided that $1/\gamma$ grows at a polynomial rate at most; see Theorem \ref{TH1}.
 	\item When $n=3$, we show  uniform-in-time boundedness of global classical solutions supposing additionally that $1/\gamma$ grows at most linearly in $v$; see Theorem \ref{TH3d}.
 	\item For the case $\gamma(v)=e^{-v}$ and $n=2,$  classical solution always exists globally due to our first main result. Besides, we show that the solution is uniformly-in-time bounded if the total mass is less than some critical mass $\Lambda_c>0$ while with certain initial data of super-critical mass, we verify occurence of  inifinite-time blowup; see Theorem \ref{TH4}.
 \end{enumerate}

Now, let us sketch the idea of our comparison method in deriving the upper bound of $v$, which is the main novelty of the present contribution.  First, inspired by our previous work \cite{FJ19}, we introduce a non-negative auxiliary function $w(x,t)$ which is the solution of the following elliptic Helmholtz equation:
\begin{equation}
\begin{cases}
-\Delta w+w=u &x\in\Omega,\;t>0\\
\partial_\nu w=0&x\in\partial\Omega,\;t>0.
\end{cases}
\end{equation} 
We can formally write $w(x,t)=(I-\Delta)^{-1}[u](x,t)$  and we denote $w_0(x)=(I-\Delta)^{-1}[u_0]$.  One notes that in the parabolic-elliptic case, i.e., $\varepsilon=0$ in \eqref{chemo0a}, $w$ is identical to $v$. However, in the present doubly parabolic case, from the second equation we formally have
\begin{equation}\label{vexp}
v=w-(I-\Delta)^{-1}[v_t].
\end{equation}
Thus,  it suffices to derive upper bounds for both terms on the right-hand side of \eqref{vexp}. 

To this aim, we begin with  deducing an upper bound for the auxiliary function $w$. Since we only have $L^1-$boundedness of $u$   due to the conservation of mass, the $L^\infty-$boundedness of $w$ is nontrivial.  This goal is achieved by a sutble observation of the nonlinear coupling structure and an application of comparison principle for elliptic equations. In the same manner as we have previously done in \cite{FJ19}, taking $(I-\Delta)^{-1}$ on both sides of the first equation of \eqref{chemo1}, we obtain the following key identity:
\begin{equation}\label{keyid}
\partial_tw(x,t)+u\gamma(v)=(I-\Delta)^{-1}[u\gamma(v)](x,t),
\end{equation}
which captures the intrinsic mechanism of the system.   Indeed, making use of the decreasing property of $\gamma$, thanks to the comparison principle of elliptic equations together with Gronwall's inequality, one can deduce from \eqref{keyid} that
\begin{equation*}
w(x,t)\leq w_0(x)e^{Ct},\;\;\text{for all}\;x\in\Omega\;\;\text{and}\;t\geq0
\end{equation*}
with some $C>0$ depending only on $\gamma, \Omega$ and the initial data.

The second step is to obtain an upper bound of $v-w=-(I-\Delta)^{-1}[v_t]$, where  the comparison principle for heat equations now plays a crucial role. Denote $\mathcal{L}[g]=g_t-\Delta g+g$ for any smooth function $g(x,t)$ satisfying homogeneous Neumann boundary conditions. Thanks to the key identity \eqref{keyid} again,   we are able to establish by delicate calculations that
\begin{equation*}
	\mathcal{L}[v-w]\leq \mathcal{L}[\Gamma(v)+K],\;\;\text{for all}\;x\in\Omega\;\;\text{and}\;t\geq0,
\end{equation*}
with some sufficiently large constant $K>0$ such that $v_0(x)-w_0(x)\leq \Gamma(v_0(x))+K$ for all $x\in\Omega$. Here, since  $\gamma$ has a vanishing limit, we can construct a continuous function $\Gamma(\cdot)$  such that
\begin{equation}
	\Gamma(v)\leq \varepsilon_0 v,\;\text{for all}\;v>0
\end{equation} with some $0<\varepsilon_0<1$. Then it follows directly from the comparison principle of heat equations that 
\begin{equation}
	v(x,t)\leq \frac{w(x,t)+K}{1-\varepsilon_0}
\end{equation} for all $x\in\Omega$ and $t\geq0.$ 

Our method relies on the comparison principles, which greatly differs from the energy method used in all previous literatures. The main strategy of our approach lies in the idea to compare the solution $v$ of a heat equation with an auxiliary function $w$, which is a solution of a Helmholtz  elliptic equation. To our knowledge, such an idea is used for the first time in related research and it is interesting that the application of comparison principle for elliptic equations also indispensable in the study of this fully parabolic system since we bring in the new variable $w$ satisfying an elliptic equation. Our approach makes fully use of the nonlinear coupling structure together with the decreasing property of $\gamma$ but needs no $L^p$-integrability of $u$. Morevoer, our method unveils an insight information of the nonlinear structure that degeneracy is prohibited in any finite time. This feature was firstly observed for the simplified parabolic-elliptic version of \eqref{chemo1} in our previous work \cite{FJ19} and is now verified by our comparison method  in the original fully parabolic system. Besides, we would like to stress that our  results on global existence as well as infinite-time blowup are both new for the fully parabolic system \eqref{chemo1} with asymptotically vanishing motilities since this problem has not been tackled before without any smallness assumptions or the presence of source terms.

The rest of the paper is organized as follows. In Section 2, we state our main results on problem \eqref{chemo1}. In Section 3, we provide some preliminary results and recall some useful lemmas. Then in Section 4 we use our comparison argument to derive the upper bounds of $v$. Uniform-in-time upper bounds of $v$ are also established under certain growth conditions on $1/\gamma$. Thanks to the upper bound of $v$, we are able to study global existence of classical solutions in Section 5. The last section is devoted to  the case $\gamma(v)=e^{-v}$, where the critical-mass phenomenon is proved in the two-dimensional setting.

\section{Main Results}
In this section, we state the main results cocerning global existence as well as infinite-blowup of problem \eqref{chemo1}. To begin with,  we introduce some notations and basic assumptions. Throughout this paper  we assume  that
\begin{equation}\label{ini}
(u_0,v_0)\in C^0(\overline\Omega)\times{W^{1,\infty}(\Omega)},\quad u_0\geq0,\; v_0\geq0 \quad  \mbox{in } \overline\Omega, \quad u_0\not\equiv0
\end{equation}
and for  $\gamma$ we  require
\begin{equation}\label{gamma0a}
\mathrm{(A0)}:\gamma(v)\in C^3[0,+\infty),\;\gamma(v)>0,\;\;\gamma'(v)\leq0\;\;\text{on}\;(0,+\infty).
\end{equation}
 and the following asymptotically vanishing property:
\begin{equation}\label{gamma}
\mathrm{(A1)}:\lim\limits_{s\rightarrow+\infty}\gamma(s)=0.
\end{equation}

Now we  state our first result on global existence of classical solutions in dimension two.
\begin{theorem}\label{TH1}
	Assume  $n=2$ with $\gamma(\cdot)$ satisfying $\mathrm{(A0)}$ and $\mathrm{(A1)}$.
	For any given initial data $(u_0,v_0)$ satisfying \eqref{ini}, system \eqref{chemo1} permits a unique global classical solution $(u,v)\in (C^0(\overline{\Omega}\times[0,\infty))\cap C^{2,1}(\overline{\Omega}\times(0,\infty)))^2$.
	
	In addition, if $1/\gamma$ satisfies the following growth condition:
	\begin{equation}\label{gamma2}\mathrm{(A2)}:\qquad\text{there is $k>0$ such that}
	\lim\limits_{s\rightarrow+\infty}s^{k}\gamma(s)=+\infty,
	\end{equation}then the global solution is uniformly-in-time bounded.
\end{theorem}
\begin{remark}
	The above result still holds true if one replaces assumption $\mathrm{(A1)}$ by the following
\begin{equation}\label{gamma1}
\mathrm{(A1')}:\lim\limits_{s\rightarrow+\infty}\gamma(s)=\gamma_\infty<1.
\end{equation}	
\end{remark}
\begin{remark}
If $v_0>0$ in $\overline{\Omega}$,	thanks to the positive time-independent lower bound $v_*$ of $v$ for $(x,t)\in\overline{\Omega}\times[0,\infty)$ given in Lemma \ref{lowbound} in the next section, our existence and boundedness results also hold true if $\gamma(s)$ has singluarities at $s=0$, for example $\gamma(s)=s^{-k}$ with $k>0$. In such cases, we can simply replace $\gamma(s)$ by a new motility function $\tilde{\gamma}(s)$ which satisfies $\mathrm{(A0)}$ and coincides with $\gamma(s)$ for $s\geq\frac{v_*}{2}$.
\end{remark}
\begin{remark} 	
	 Our result generalizes the corresponding boundedness result in \cite{Anh19} established for the simplified parabolic-elliptic system with special motility $v^{-k}$ with any $k>0$ to more general functions satisfying  $\mathrm{(A0)}$,  $\mathrm{(A1)}$ and  $\mathrm{(A2)}$, for example, $\gamma(v)=\frac{1}{v^k\log(1+v)}$ with any $k>0.$
\end{remark}
\begin{remark}\label{rem_anytau}
Theorem \ref{TH1} is independent of the coefficients of the system. In particular, if the second equation of \eqref{chemo1} is replaced with 
$$
\tau v_t = \Delta v - v +u
$$
with $\tau>0$, Theorem \ref{TH1} is still valid for any $\tau>0$. See Remark \ref{compwv_tau}, Remark \ref{unifbddv_tau} and Remark \ref{rem_tau}.
\end{remark}
\begin{remark}
In the case $\gamma(v)=v^{-k}$ with $k>0$, the variant form reads
	\begin{equation}\label{variant2}
	u_t=\nabla\cdot\left[\gamma(v)(\nabla u-ku\nabla \log v)\right], 
	\end{equation}
	which resembles the classical Keller--Segel model with a logarithmic chemo-sensitivity:
\begin{equation}\label{logKS}
\begin{cases}
u_t=\nabla\cdot (\nabla u-ku\nabla \log v),\\
\tau v_t=\Delta v-v+u.
\end{cases}
\end{equation}
Indeed, they have the same stationary problem. 
As to the two dimensional Keller--Segel model with a logarithmic chemo-sensitivity, 
global existence and uniform-in-time boundedness of solutions were established for sufficiently small or sufficiently large $\tau>0$ in \cite{FS2016, fs2018}. 
Even global existence of solutions for any $\tau>0$ is still open.
On the other hand, Remark \ref{rem_anytau} claims global existence and uniform-in-time  boundedness of solutions to \eqref{chemo1} for any $\tau>0$. 
\end{remark}
In the three-dimensional case, we obtain existence of uniformly-in-time bounded classical solution with a stronger growth condition on $1/\gamma$.
\begin{theorem}\label{TH3d}
	Assume $n=3$ and $\gamma(\cdot)$ satisfies  $\mathrm{(A0)}$, $\mathrm{(A1)}$ and additionally 
	\begin{equation}\label{gamma3}
	\mathrm{(A3)}:\qquad	2|\gamma'(s)|^2\leq\gamma(s)\gamma''(s),\;\;\forall\;s>0.
	\end{equation}	For any given initial data $(u_0,v_0)$ satisfying \eqref{ini}, system \eqref{chemo1} permits a unique global classical solution $(u,v)\in (C^0(\overline{\Omega}\times[0,\infty))\cap C^{2,1}(\overline{\Omega}\times(0,\infty)))^2$ which is uniformly-in-time bounded.
\end{theorem}
\begin{remark}Note that $\mathrm{(A3)}$ is a more restrictive growth condition than $\mathrm{(A2)}$. Under assumptions $\mathrm{(A0)}$, $\mathrm{(A1)}$ and $\mathrm{(A3)}$, $1/\gamma(s)$ can grow at most linearly in $s$; see Lemma \ref{lemA23}.

	In fact when $n=3$, we can establish uniform-in-time boundedness of $v$ with  $\gamma(\cdot)$ satisfying  $\mathrm{(A0)}$,  $\mathrm{(A1)}$ and $\mathrm{(A2)}$ with any $0<k<2$. However, for technique reasons, we can now only achieve uniform-in-time bounds of $u$ with the help of assumption  $\mathrm{(A3)}$; see Section 5.3 for more details. 
\end{remark}
\begin{remark}	
	When $n=3$ and $\gamma(v)=v^{-k}$ with $k>0$,  $\mathrm{(A3)}$ is equivalent to a constraint $0<k\leq1.$ 
Comparing with the Keller--Segel model with a logarithmic chemo-sensitivity \eqref{logKS}, 
 the condition$\mathrm{(A3)}$ reduces to a restriction on the chemo-sensitivity coefficient $k$. 
Global existence of \eqref{logKS} is still open for large $k$ when $n\geq3$. 
We refer the readers to \cite{BBTW15, fs2018} for reviews of related topics.
\end{remark}

Last, we verify the following critical mass phenomenon for the case $\gamma(v)=e^{-v}$.
\begin{theorem}\label{TH4}
	Assume $n=2$, $\gamma(v)=e^{-v}$ and $(u_0,v_0)$ satisfies \eqref{ini}. Let
	\begin{equation}
	\Lambda_c=\begin{cases}
	8\pi\qquad\text{if}\;\Omega=B_R(0)\triangleq\{x\in\mathbb{R}^2;\;|x|<R\}\;\;\text{with}\;R>0\;\text{and}\;(u_0,v_0) \;\text{is radial in}\;x,\\
	4\pi\qquad\text{otherwise.}		\nn
	\end{cases}
	\end{equation} Then if $\Lambda\triangleq\int_\Omega u_0 dx<\Lambda_c$, the global classical solution of \eqref{chemo2a} is uniformly-in-time bounded. Moreover, the solution converges to an equilibrium as time goes to infinity, i.e., there is a solution $(u_s,v_s)$ to the stationary problem \eqref{steady}, such that
	\[
\lim_{t \rightarrow +\infty} (u(t),v(t)) = (u_s,v_s) \quad \mbox{{\rm in} }C^2(\overline{\Omega}).
\]

	On the other hand, there exists non-negative initial datum $(u_0,v_0)$ satisfying \eqref{ini} with $\Lambda\in(8\pi,\infty)\backslash4\pi\mathbb{N}$ such that the corresponding global classical solution blows up at time infinity. More precisely,
	\begin{equation*}
	\lim\limits_{t\nearrow +\infty}\|u(\cdot,t)\|_{L^\infty(\Omega)}=\limsup\limits_{t\nearrow +\infty}\|(I-\Delta)^{-1}[u](\cdot,t)\|_{L^\infty(\Omega)}=\limsup\limits_{t\nearrow +\infty}\|v(\cdot,t)\|_{L^\infty(\Omega)}=+\infty.
	\end{equation*}
\end{theorem}

\section{Preliminaries}
In this section, we recall some useful lemmas. First, local existence and uniqueness of classical solutions to system \eqref{chemo1} can be
established by the standard fixed point argument and  regularity theory for parabolic equations. Similar proof can be found in \cite[Lemma 3.1]{Anh19} or \cite[Lemma 2.1]{JKW18}  and hence here we omit the detail here.
\begin{theorem}\label{local}
	Let $\Omega$ be a smooth bounded domain of $\mathbb{R}^n$. Suppose that $\gamma(\cdot)$ satisfies \eqref{gamma0a} and $(u_0,v_0)$ satisfies \eqref{ini}. Then there exists $T_{\mathrm{max}} \in (0, \infty]$ such that problem \eqref{chemo1} permits a unique non-negative classical solution $(u,v)\in (C^0(\overline{\Omega}\times[0,T_{\mathrm{max}}))\cap C^{2,1}(\overline{\Omega}\times(0,T_{\mathrm{max}})))^2$. Moreover, the following mass conservation holds
	\begin{equation*}
	\int_{\Omega}u(\cdot,t)dx=\int_{\Omega}u_0 dx
	\quad \text{for\ all}\ t \in (0,T_{\mathrm{max}}).
	\end{equation*}	
	If $T_{\mathrm{max}}<\infty$, then
	\begin{equation*}
	\lim\limits_{t\nearrow T_{\mathrm{max}}}\|u(\cdot,t)\|_{L^\infty(\Omega)}=\infty.
	\end{equation*}

\end{theorem}

Next, we recall the following lemma given in \cite{BS,Anh19} about estimates for the solution of Helmholtz equations. Let $a_+=\max\{a,0\}$. Then we have
\begin{lemma}\label{lm2}
	Let $\Omega$ be  a smooth bounded domain in $\mathbb{R}^n$, $n\geq1$ and let $f\in C(\overline{\Omega})$ be a non-negative  function such that $\int_\Omega f dx>0$. If $z$ is a $C^2(\overline{\Omega})$ solution to
	\begin{equation}
	\begin{split}
		-\Delta z+z=f,\;\;x\in\Omega,\\
		\frac{\partial z}{\partial \nu}=0\;\;x\in\partial\Omega,
	\end{split}
	\end{equation}then if $1\leq q< \frac{n}{(n-2)_+}$, there exists a positive constant $C=C(n,q,\Omega)$ such that
	\begin{equation}
		\|z\|_{L^q(\Omega)}\leq C\|f\|_{L^1(\Omega)}.
	\end{equation}
\end{lemma}


A strictly positive uniform-in-time lower bound for $v$ was given in \cite[Lemma 2.1]{FS2016} provided that $v_0$ is strictly positive in $\overline{\Omega}$.
\begin{lemma}\label{lowbound}
	Assume that $(u_0,v_0)$ satisfies \eqref{ini} and moreover $v_0>0$ in $\overline{\Omega}$. If $(u,v)$ is the solution of \eqref{chemo1} in $\Omega \times (0,T)$,
	then  there exists some $v_* >0$ such that
	\begin{eqnarray*}
		\inf_{x\in \Omega} v(x,t ) \geq v_*>0\qquad
		\mbox{for all }t\in(0,T).
	\end{eqnarray*}
	Here the constant $v_*$ is independent of $T>0$.
\end{lemma}

Then, we recall the following lemma given in \cite[Lemma 2.4]{FS2016}.
\begin{lemma}\label{fs}Let $n=2$ and $p\in(1,2)$. There exists $K_{Sob}>0$ such that
	for all $s>1$ and for all $t\in [0,T_{\mathrm{max}})$,
	\begin{align*}
	\io u^{p+1}
	\leq
	\dfrac{{K_{Sob}(p+1)}^2}{\log s}\io (u\log u +e^{-1})
	\io u^{p-2} {|\nabla u|}^2
	+6s^{p+1} |\Omega| +4K_{Sob}^2{|\Omega|}^{2-p}\|u_0\|^{p+1}_{L^1(\Omega)}.
	\end{align*}
\end{lemma}
In addition, we need  the following uniform Gronwall inequality \cite[Chapter III, Lemma 1.1]{Temam} to deduce uniform-in-time estimates for the solutions.
\begin{lemma}\label{uniformGronwall}
	Let $g,h,y$ be three positive locally integrable functions on $(t_0,\infty)$ such that $y'$ is locally integrable on $(t_0,\infty)$ and the following inequalities are satisfied:
	\begin{equation*}
	y'(t)\leq g(t)y(t)+h(t)\;\;\forall\;t\geq t_0,
	\end{equation*}
	\begin{equation*}
	\int_t^{t+r}g(s)ds\leq a_1,\;\;\int_t^{t+r}h(s)ds\leq a_2,\;\;\int_t^{t+r}y(s)ds\leq a_3,\;\;\forall \;t\geq t_0
	\end{equation*}	where $r,a_i$, $(i=1,2,3)$ are positive constants. Then
	\begin{equation*}
	y(t+r)\leq \left(\frac{a_3}{r}+a_2\right)e^{a_1},\;\;\forall t\geq t_0.
	\end{equation*}
\end{lemma}

\section{The Comparison Method and the Upper Bound of $v$}
In this section, we establish the upper bounds of $v$ by our comparison method as illustrated in the Introduction. To begin with, we define an auxiliary variable $w(x,t)$, which is the unique non-negative solution of the following Helmholtz equation:
\begin{equation*}
\begin{cases}
-\Delta w+w=u, &x\in\Omega,\;t>0,\\
\partial_\nu w=0,&x\in\partial\Omega,\;t>0.
\end{cases}
\end{equation*} 
Then we derive the key identity and establish a point-wise upper bound  for $w$ as follow. 
Here and in the sequel, $v_*=0$ if $v_0\geq0$  and $v_*>0$ if $v_0>0$ in $\overline{\Omega}$ due to Lemma \ref{lowbound}.

\begin{lemma}\label{keylem1}Assume $n\geq1$. For any $0<t<T_{\mathrm{max}}$, there holds
	\begin{equation}\label{var0}
	w_t+\gamma(v)u=(I-\Delta)^{-1}[\gamma(v)u].
	\end{equation}
	Moreover, for  any $x\in\Omega$ and  $t\in[0,T_{\mathrm{max}})$, we have
	\begin{equation}\label{ptesta}
		w(x,t)\leq w_0(x)e^{\gamma(v_*)t}.
	\end{equation}
\end{lemma}
\begin{proof} The proof was already given in our previous paper \cite{FJ19}. For the completeness of the present work, we report in detail here.
	First,  the key identity \eqref{var0} follows by taking $(I-\Delta)^{-1}$ on both sides of the first equation in \eqref{chemo1}. Here, $\Delta$ is the Laplacian operator with homogeneous Neumann boundary conditions.
	 
Note that $v$ is non-negative due to the maximum principle of heat equations.    Since $\gamma$ is non-increasing in $v$, there holds $\gamma(v)\leq \gamma(v_*)$ for all $(x,t)\in\Omega\times[0,T_\mathrm{max})$. 
As a result, we infer by comparison principle of elliptic equations that for any $(x,t)\in\Omega\times [0,T_{\mathrm{max}})$, \begin{equation*}(I-\Delta)^{-1}[\gamma(v)u]\leq (I-\Delta)^{-1}[\gamma(v_*)u]=\gamma(v_*)w\end{equation*} and it follows from \eqref{var0} that
\begin{equation}\label{var0b}
w_t+\gamma(v)u\leq \gamma(v_*) w.
\end{equation} 
Since $\gamma(v)u\geq0$,  an application of Gronwall's inequality together with \eqref{var0b} gives rise to
	\begin{equation*}
w(x,t)\leq w_0(x)e^{\gamma(v_*)t},
\end{equation*} 	
which completes the proof.
\end{proof}

Next, we aim to compare $v$ with the bounded auxiliary function $w$. 
Observing that $\lim\limits_{s\rightarrow+\infty}\gamma(s)=0$, we can fix some $a>0$ such that $0<\gamma(a)<1$ and for any $s\geq0$ we define
\begin{equation*}
\Gamma(s)=\int_a^s\gamma(\eta)d\eta.
\end{equation*}
Then, one can easily verify the following relation between $\gamma$ and $\Gamma.$
\begin{lemma}\label{lemGam}
	Under the assumption of $(\mathrm{A0})$ and $(\mathrm{A1})$, for any $s_0\in[0,a)$ there is $C_a(s_0)>0$ depending on $a$ and $s_0$ such that 
	\begin{equation}\label{Gamma0}
	s\gamma(s)-C_a(s_0)\leq \Gamma(s)\leq \gamma(a)s,\;\;\forall\;s\geq s_0.
	\end{equation}
\end{lemma}
\begin{proof}
	First, we assert that there is $C_a>0$ depending on $a$ such that
		\begin{equation}\label{Gamma1}
	s\gamma(s)-C_a\leq \Gamma(s)\leq \gamma(a)s,\;\;\forall\;s\geq a.
	\end{equation}
Indeed, by Taylor expansion we infer that
	\begin{equation}
	\Gamma(s)=\gamma(a)(s-a)+\frac12\gamma'(a\theta+s(1-\theta))(s-a)^2,\;\;\text{for some}\;\theta\in(0,1).
	\end{equation}
	Then due to the fact $\gamma'\leq0$, we obtain that
	\begin{equation*}
	\Gamma(s)\leq \gamma(a)(s-a),
	\end{equation*}
	which yields the most right-hand side of \eqref{Gamma1}.
	
	On the other hand,	since $\gamma$ is decreasing, we infer that for $s\geq a$,
	\begin{equation*}
	\Gamma(s)=\int_a^s\gamma(\eta)d\eta\geq\gamma(s)(s-a)=s\gamma(s)-a\gamma(s).
	\end{equation*}
	Therefore, when $s\geq a$, using the fact  $\gamma(s)\leq \gamma(a)$,
	\begin{equation}
	s\gamma(s)\leq \Gamma(s)+a\gamma(s)\leq \Gamma(s)+a\gamma(a).
	\end{equation}

	Then in order to establish \eqref{Gamma0}, it remains to check the case $s_0\leq s\leq a$. The most right-hand side is trivial since $\Gamma(s)\leq 0$ by definition when $s_0\leq s\leq a$. On the other hand  when $s_0\leq s\leq a$, using the decreasing property of $\gamma$ again, there holds
	\begin{equation}
	\begin{split}
	s\gamma(s)-\Gamma(s)=&s\gamma(s)+\int_s^a\gamma(\eta)d\eta\\
	\leq& a\gamma(s_0)+\gamma(s_0)(a-s_0)\\
	\leq& 2a\gamma(s_0),
	\end{split}
	\end{equation}	
	which completes the proof.
\end{proof}
Now, we are ready to apply the comparison principle of parabolic equations to obtain the following result.
\begin{lemma}\label{vbd}
	Under the assumption of $(\mathrm{A0})$ and $(\mathrm{A1})$, there is $K>0$ depending on $a$ and the initial data such that for all $(x,t)\in\Omega\times[0,T_{\mathrm{max}})$,
	\begin{equation}\label{vbound}
		v(x,t)\leq \frac{1}{1-\gamma(a)}\bigg(w(x,t)+K\bigg).
	\end{equation}
\end{lemma}
\begin{proof}
	 Recall that $w-\Delta w=u$. Substituting the key identity  \eqref{var0} into the second equation of \eqref{chemo1}, we observe that
\begin{equation}
	\begin{split}\label{compare0}
	v_t-\Delta v+v=&w-\Delta w\\
	=&w-\Delta w+w_t-w_t\\
	=&w_t-\Delta w+w+\gamma(v)u-(I-\Delta)^{-1}[\gamma(v)u].
	\end{split}
\end{equation}
Using the second equation of \eqref{chemo1} again, we observe that
\begin{equation}\label{Gamma}
	\begin{split}
	\gamma(v)u=&\gamma(v)(v_t-\Delta v+v)\\
	=&\bigg(\partial_t\Gamma(v)-\Delta \Gamma(v)+\Gamma(v)\bigg)+\gamma'(v)|\nabla v|^2+v\gamma(v)-\Gamma(v).
	\end{split}
\end{equation}
Then plugging \eqref{Gamma} into \eqref{compare0} yields that
\begin{equation}
	\begin{split}\label{comp0}
	&v_t-\Delta v+v+(I-\Delta)^{-1}[\gamma(v)u]-\gamma'(v)|\nabla v|^2\\
	&=\bigg(\partial_t(w+\Gamma(v))-\Delta (w+\Gamma(v))+\left(w+\Gamma(v)\right)\bigg)+\left(v\gamma(v)-\Gamma(v)\right).
	\end{split}
\end{equation}
According to Lemma \ref{lemGam}, there  is $C(v_*)>0$ depending on $a$ and $v_*$ such that for all $(x,t)\in\Omega\times[0,T_{\mathrm{max}})$
\begin{equation}
	v\gamma(v)-\Gamma(v)\leq C(v_*).
\end{equation}
In addition, since $(I-\Delta)^{-1}[\gamma(v)u]$ and $-\gamma'(v)|\nabla v|^2$ are both non-negative, it follows from \eqref{comp0} that for all $(x,t)\in\Omega\times[0,T_{\mathrm{max}})$,
\begin{equation}
	v_t-\Delta v+v
	\leq \bigg(\partial_t(w+\Gamma(v))-\Delta (w+\Gamma(v))+\left(w+\Gamma(v)\right)\bigg)+C(v_*).
\end{equation}
Now, in view of our assumption \eqref{ini} on the initial data, we may choose a positive constant $K\geq C(v_*)$ such that $v_0\leq w_0+\Gamma(v_0)+K$ for all $x\in\overline{\Omega}$. Then we deduce by comparison principle for heat equations that 
\begin{equation}
	v(x,t)\leq w(x,t)+\Gamma(v(x,t))+K,\;\;\forall(x,t)\in\Omega\times[0,T_{\mathrm{max}}).
\end{equation}
Finally, we may conclude the proof with the fact that 
\begin{equation*}
	\Gamma(v(x,t))\leq \gamma(a)v(x,t)
\end{equation*} due to Lemma \ref{lemGam} again.
\end{proof}

\begin{remark}\label{compwv_tau}
The similar result of Lemma \ref{vbd} still holds true 
if we replace the second equation of \eqref{chemo1} by 
$$
\tau v_t = \Delta v - v +u
$$
with a constant $\tau>0$. 
Indeed, one can give a suitable modification as follows.
For fixed $\tau>0$,
 we can choose some $a>0$ such that $0<\gamma(a)<\frac{1}{\tau}$ due to the assumption $\lim\limits_{s\rightarrow+\infty}\gamma(s)=0$. 
 With the function $\Gamma$ which is defined by the above $a>0$, 
we proceed the similar lines as
\begin{equation*}
	\begin{split}
	\tau v_t-\Delta v+v=&w-\Delta w\\
	=&\tau w_t-\Delta w+w+\tau \left(\gamma(v)u-(I-\Delta)^{-1}[\gamma(v)u]\right),
	\end{split}
\end{equation*}
and
\begin{equation*}
	\begin{split}
	\tau \gamma(v)u=&\tau \gamma(v)(\tau v_t-\Delta v+v)\\
	=&\bigg(\tau \partial_t (\tau\Gamma(v))-\Delta (\tau \Gamma(v))+(\tau\Gamma(v)) \bigg)+\tau \gamma'(v)|\nabla v|^2+\tau v\gamma(v)-\tau \Gamma(v),
	\end{split}
\end{equation*}
thus we derive
\begin{equation*}
	\begin{split}
	&\tau v_t-\Delta v+v+\tau (I-\Delta)^{-1}[\gamma(v)u]-\tau\gamma'(v)|\nabla v|^2\\
	&=\bigg(\tau \partial_t(w+\tau \Gamma(v))-\Delta (w+\tau \Gamma(v))+\left(w+\tau \Gamma(v)\right)\bigg)
	+\tau \left(v\gamma(v)-\Gamma(v)\right).
	\end{split}
\end{equation*}
By the same discussion, for any $(x,t)\in\Omega\times[0,T_{\mathrm{max}})$ we have
\begin{equation*}
	v(x,t)\leq w(x,t)+\tau \Gamma(v(x,t))+K
	 \leq w(x,t)+\tau\gamma(a)v(x,t)+K,
\end{equation*}
which implies
	\begin{equation*}
		v(x,t)\leq \frac{1}{1-\tau \gamma(a)}\bigg(w(x,t)+K\bigg).
	\end{equation*}
\end{remark}

Next, we establish uniform-in-time boundedness of $v$ with the growth condition $(\mathrm{A2})$ on $1/\gamma$.
\begin{lemma}\label{lm41}
	Assume $n=2,3$. Then under the assumptions  $\mathrm{(A0)}$,  $\mathrm{(A1)}$ and  $\mathrm{(A2)}$ with $0<k<\frac{2}{(n-2)_+},$ there exists $C>0$ depending only on $\gamma$, $\Omega$ and the initial data such that for all $(x,t)\in\Omega\times[0,T_{\mathrm{max}})$,
	\begin{equation*}
	v(x,t)\leq C.
	\end{equation*}
\end{lemma}
\begin{proof}
	Multiplying the first equation of \eqref{chemo1} by $w=(I-\Delta)^{-1}[u]$ and integrating over $\Omega$, we obtain that
	\begin{equation}
	\frac{1}{2}\frac{d}{dt}(\|\nabla w\|_{L^2(\Omega)}^2+\|w\|_{L^2(\Omega)}^2)+\int_\Omega \gamma(v)u^2dx=\int_\Omega \gamma(v)u wdx.\nn
	\end{equation}	
	Thanks to the fact that $\gamma(v)\leq\gamma (v_*)$, we obtain that
	\begin{equation}\label{uni0}
	\frac{1}{2}\frac{d}{dt}(\|\nabla w\|_{L^2(\Omega)}^2+\|w\|_{L^2(\Omega)}^2)+\int_\Omega \gamma(v)u^2dx\leq \Lambda\gamma(v_*)\|w\|_{L^\infty(\Omega)},
	\end{equation}where $\Lambda=\int_\Omega u_0dx.$
	On the other hand,  by integration by parts and Young's inequality, we infer that
	\begin{align}\label{add0}
	\nn
	\|\nabla w\|_{L^2(\Omega)}^2+\|w\|_{L^2(\Omega)}^2=&\int_\Omega w udx\non\\
	\leq&\int_\Omega \gamma(v)u^2dx+\int_\Omega\gamma^{-1}(v)w^2dx.
	\end{align}
	In view of our assumption (A2), we may infer that there exist $k\in (0,\frac{2}{(n-2)_+} )$, $b>0$ and $s_b>v_*$ such that for all $s\geq s_b$
	\begin{equation*}
	\gamma^{-1}(s)\leq bs^k
	\end{equation*}and on the other hand, since $\gamma(\cdot)$ is decreasing,
	\begin{equation*}
	\gamma^{-1}(s)\leq \gamma^{-1}(s_b)
	\end{equation*}for all $0\leq s<s_b$.
	Therefore, for all $s\geq0$, there holds
	\begin{equation}\label{cond_gamma}
	\gamma^{-1}(s)\leq bs^{k}+\gamma^{-1}(s_b).
	\end{equation}
	Therefore, we deduce from above and Lemma \ref{vbd} that
	that	\begin{equation}\label{unic}
	\begin{split}
	\int_\Omega\gamma^{-1}(v)w^2dx\leq &\int_\Omega (bv^k+\gamma^{-1}(s_b))w^2dx\\
	\leq&\int_\Omega \left(b\left(\frac{1}{1-\gamma(a)}(w+K)\right)^k+\gamma^{-1}(s_b)\right)w^2dx\\
	\leq &C\int_\Omega w^{k+2}dx+C
	\end{split}
	\end{equation}
	with $C>0$ depending only on the initial data, $\gamma$ and $\Omega$.
	
	On the other hand, for any $\frac{n}{2}<p<2$, due to the Sobolev embedding theorem and H\"older's inequality, we have
	\begin{equation}
	\begin{split}
	\|w\|_{L^\infty(\Omega)}\leq&  C\|u\|_{L^p(\Omega)}\\
	\leq& C\left(\int_\Omega \gamma(v)u^2dx\right)^{1/2}\left(\int_\Omega \gamma^{-\frac{p}{2-p}}(v)dx\right)^{\frac{2-p}{2p}}.
	\end{split}
	\end{equation}	
In the same manner as before, we infer that
	\begin{equation}
	\begin{split}
	\int_\Omega \gamma^{-\frac{p}{2-p}}(v)dx\leq & \int_\Omega \left(b v^{k}+\gamma^{-1}(s_b)\right)^{\frac{p}{2-p}}dx\\
	\leq&\int_\Omega \left(b\left(\frac{1}{1-\gamma(a)}(w+K)\right)^k+\gamma^{-1}(s_b)\right)^{\frac{p}{2-p}}dx\\
	\leq &C\int_\Omega w^{\frac{pk}{2-p}}dx+C,
	\end{split}
	\end{equation} where $C>0$ depending only on the initial data, $\gamma$ and $\Omega$.
	Thus, by Young's inequality with any $\delta>0$, there holds
	\begin{equation}
	\begin{split}\label{winf}
	\|w\|_{L^\infty(\Omega)}\leq& \delta\int_\Omega \gamma(v)u^2dx+C_{\delta}\left(\int_\Omega \gamma^{-\frac{p}{2-p}}(v)dx\right)^{\frac{2-p}{p}}\\
	\leq& \delta\int_\Omega \gamma(v)u^2dx+C_{\delta}\left(\int_\Omega w^{\frac{pk}{2-p}}dx\right)^{\frac{2-p}{p}}+C_{\delta}^\prime.
	\end{split}
	\end{equation}
	As a result, we deduce from preceding inequalities \eqref{uni0}, \eqref{add0}, \eqref{unic} and \eqref{winf} that
	\begin{equation}
	\begin{split}\label{w1}
	&\frac{d}{dt}(\|\nabla w\|_{L^2(\Omega)}^2+\|w\|_{L^2(\Omega)}^2)+\int_\Omega \gamma(v)u^2dx+(\|\nabla w\|_{L^2(\Omega)}^2+\|w\|_{L^2(\Omega)}^2)\\
	&\leq 2\delta\Lambda \gamma(v_*)\int_\Omega\gamma(v)u^2dx+2C_\delta\Lambda \gamma(v_*)\left(\int_\Omega w^{\frac{pk}{2-p}}dx\right)^{\frac{2-p}{p}}+2C\int_\Omega w^{k+2}dx+C_\delta.
	\end{split}
	\end{equation}
	
	Next, we divide our argument into two cases. First, when $n=2$, recalling that $w=(I-\Delta)^{-1}[u]$ and thanks  to Lemma \ref{lm2},  we have
	\begin{equation}
	\left(\int_\Omega w^{\frac{pk}{2-p}}dx\right)^{\frac{2-p}{p}}+\int_\Omega w^{k+2}dx\leq C
	\end{equation} with some $C>0$ depending only on $\Omega$ and $\|u_0\|_{L^1(\Omega)}$.
	As a result, for $n=2$, by picking small $\delta>0$ in \eqref{w1}, we obtain that
	\begin{equation}\label{wn2}
	\frac{d}{dt}(\|\nabla w\|_{L^2(\Omega)}^2+\|w\|_{L^2(\Omega)}^2)+\frac12\int_\Omega \gamma(v)u^2dx+(\|\nabla w\|_{L^2(\Omega)}^2+\|w\|_{L^2(\Omega)}^2)\leq C,
	\end{equation} which by means of ODE analysis yields that
	\begin{equation}
	\|\nabla w\|_{L^2(\Omega)}^2+\|w\|_{L^2(\Omega)}^2\leq C
	\end{equation}with $C>0$ depending only on the initial data, $\gamma$ and $\Omega$. Moreover, it follows from \eqref{winf} and \eqref{wn2} that for any $t\in(0,T_{\mathrm{max}}-\tau)$ with  $\tau=\min\{1,\frac12T_{\mathrm{max}}\}$,
	\begin{equation}
	\int_t^{t+\tau}\|w\|_{L^\infty(\Omega)}ds\leq C\int_t^{t+\tau}\int_\Omega \gamma(v)u^2dxds+C\leq C.
	\end{equation}
	
	On the other hand, when $n=3$, for any $1\leq q<3$ and $3\leq r\leq 6$, we recall the Gagliardo-Nirenberg inequality
	\begin{equation*}
	\|w\|_{L^r(\Omega)}\leq C\|\nabla w\|_{L^2(\Omega)}^{\beta}\|w\|_{L^{q}(\Omega)}^{1-\beta}+C\|w\|_{L^1(\Omega)}.
	\end{equation*}
	with 
	\begin{equation*}
	\beta=(\frac1q-\frac1r)/(\frac1q-\frac16)\in(0,1].
	\end{equation*}
	Since $\|w\|_{L^q(\Omega)}$ with $1\leq q<3$ is bounded due to Lemma \ref{lm2}, we infer that for any $k\leq 4$
	\begin{equation*}
	\int_\Omega w^{k+2}dx\leq C\|\nabla w\|^{\beta_1(k+2)}+C
	\end{equation*}where
	\begin{equation*}
	\beta_1=(\frac{1}{q_1}-\frac{1}{k+2})/(\frac1{q_1}-\frac16)
	\end{equation*}
	and for any $\frac{pk}{2-p}\leq 6$ with some $\frac32<p<2$,
	\begin{equation*}
	\left(\int_\Omega w^{\frac{pk}{2-p}}dx\right)^{\frac{2-p}{p}}\leq C\|\nabla w\|^{k\beta_2}+C
	\end{equation*}where
	\begin{equation*}
	\beta_2=(\frac{1}{q_2}-\frac{2-p}{pk})/(\frac1{q_2}-\frac16).
	\end{equation*}
	We further require that $\beta_1(k+2)<2$ as well as $k\beta_2<2$ and then collecting the above inequalities on parameters, we get
	\begin{equation}
	\begin{cases}
	1\leq q_1,q_2<3,\\
	\frac{pk}{2-p}\leq 6,\\
	\frac32<p<2,\\
	0<k\leq 4,\\
	\beta_1(k+2)<2,\\
	k\beta_2<2.\\
	\end{cases}
	\end{equation}
	Then a direct calculation implies that for any $0<k<2$, we can find $p,q_1,q_2$ satisfying the above relations such that
	\begin{equation}
	\left(\int_\Omega w^{\frac{pk}{2-p}}dx\right)^{\frac{2-p}{p}}+\int_\Omega w^{k+2}dx\leq C\|\nabla w\|^\zeta+C
	\end{equation}
	with some $0<\zeta<2$. Now, we may use Young's inequality in \eqref{w1} to obtain that
	\begin{equation}\label{wn3}
	\frac{d}{dt}(\|\nabla w\|_{L^2(\Omega)}^2+\|w\|_{L^2(\Omega)}^2)+\frac12\int_\Omega \gamma(v)u^2dx+\frac12(\|\nabla w\|_{L^2(\Omega)}^2+\|w\|_{L^2(\Omega)}^2)\leq C
	\end{equation} where $C>0$ depends only on $\gamma$, $\Omega$ and the initial data. Then in the same manner as before, we obtain that
	\begin{equation}\label{ubw1}
	\|\nabla w\|_{L^2(\Omega)}^2+\|w\|_{L^2(\Omega)}^2\leq C
	\end{equation}
	and for any $t\in(0,T_{\mathrm{max}}-\tau)$ with  $\tau=\min\{1,\frac12T_{\mathrm{max}}\}$,
	\begin{equation}\label{ubw2}
	\int_t^{t+\tau}\|w\|_{L^\infty(\Omega)}\leq C \int_t^{t+\tau}\int_\Omega\gamma(v)u^2dxds+C\leq C.
	\end{equation}

	In summary, we establish uniform-in-time bounds \eqref{ubw1} and \eqref{ubw2} for $n=2,3$ with any $0<k<\frac{n}{(n-2)_+}$, which in particular indicates that for any fixed $x\in\Omega$ and any $t\in(0,T_{\mathrm{max}}-\tau)$ with  $\tau=\min\{1,\frac12T_{\mathrm{max}}\}$,
	\begin{equation}
	\int_t^{t+\tau}w(x,s)ds\leq	\int_t^{t+\tau}\|w\|_{L^\infty(\Omega)}\leq C.
	\end{equation}
	Then, we recall that
	\begin{equation*}
	w_t+\gamma(v)u=(I-\Delta)^{-1}[\gamma(v)u]\leq \gamma(v_*)w.
	\end{equation*}
	With the aid of the uniform Gronwall inequality Lemma \ref{uniformGronwall}, we infer for any $x\in\Omega$ and $t\in(\tau,T_{\mathrm{max}})$
	\begin{equation}
	w(x,t)\leq C
	\end{equation}with some $C>0$ independent of $x$, $t$ and $T_{\mathrm{max}}$ which together with Lemma \ref{keylem1} for $t\leq \tau$ gives rise to the uniform-in-time boundedness of $w$ such that for all $(x,t)\in\Omega\times[0,T_{\mathrm{max}}),$
	\begin{equation}
	w(x,t)\leq C.
	\end{equation}
	This concludes the proof due to Lemma \ref{vbd} since
	\begin{equation*}
	v(x,t)\leq \frac{1}{1-\gamma(a)}\bigg(w(x,t)+K\bigg).
	\end{equation*}
\end{proof}

\begin{remark}
	The results of Lemma \ref{lemGam}, Lemma \ref{vbd} and Lemma \ref{lm41} still hold ture if one replaces the assumption $(\mathrm{A1})$ by the following
	\begin{equation}
	\mathrm{(A1')}:\lim\limits_{s\rightarrow+\infty}\gamma(s)=\gamma_\infty<1.
	\end{equation}
\end{remark}

\begin{remark}\label{unifbddv_tau}
In light of Lemma \ref{compwv_tau}, the result of Lemma \ref{lm41} still holds if we replace the second equation of \eqref{chemo1} by 
$$
\tau v_t = \Delta v - v +u
$$
with a constant $\tau>0$. 
\end{remark}

\section{Existence and Boundedness of Classical Solutions}
In this section, we prove Theorem \ref{TH1} and Theorem \ref{TH3d} via the classical energy method.

\subsection{A Priori Estimates}
 To begin with, we derive some energy estimates.
\begin{lemma}\label{est0}Assume $n\geq1$. There exists $C>0$ depending on the $\|u_0\|_{L^1(\Omega)}$ and $\Omega$ such that for any $t\in[0,T_{\mathrm{max}})$,
	\begin{equation}\label{L2u}
	\|u(t)-\overline{u_0}\|_{H^{-1}}^2+\|w(t)\|_{H^1(\Omega)}^2+\int_0^t\int_\Omega \gamma(v)u^2dxds\leq 2\|u_0-\overline{u_0}\|^2_{H^{-1}(\Omega)}+2\overline{u_0}^2|\Omega| +Ct,
	\end{equation}
	where $\overline{\varphi}\triangleq\frac{1}{|\Omega|}\int_\Omega \varphi dx$ for any $\varphi\in L^1(\Omega)$.	
\end{lemma}
\begin{proof}Let $A$ denote the self-adjoint realization of $-\Delta$ under homogeneous Neumann boundary condition in the Hilbert space 
	$L^2_\perp(\Omega):=\{ \psi \in L^2(\Omega) \,|\, \io \psi=0 \}$
	with domain 
	$D(A) := \{ \psi \in H^2(\Omega) \cap L^2_\perp(\Omega) \,|\, \frac{\partial \psi}{\partial \nu}=0 \ \mbox{on }\partial \Omega \}$.
	Moreover we denote the bounded self-adjoint fractional powers $A^{-\alpha}$ with any $\alpha>0$. 
	Multiplying the first equation of \eqref{chemo1} by $A^{-1}(u-\overline{u_0})$ and integrating over $\Omega$, we obtain that
	\begin{equation}
	\frac{1}{2}\frac{d}{dt}\|A^{-\frac12}(u-\overline{u_0})\|_{L^2(\Omega)}^2+\int_\Omega \gamma(v)u^2dx=\overline{u_0}\int_\Omega \gamma(v)udx.\nn
	\end{equation}	
	Since $\gamma(v)\leq\gamma (v_*)$, we deduce that
	\begin{equation*}
	\frac{1}{2}\frac{d}{dt}\|A^{-\frac12}(u-\overline{u_0})\|_{L^2(\Omega)}^2+\int_\Omega \gamma(v)u^2dx\leq \gamma(v_*)\overline{u_0}^2|\Omega|,
	\end{equation*}which  implies  by a direct integration that for any $t\in(0,T_{\mathrm{max}})$
	\begin{equation*}
	\|A^{-\frac12}(u(t)-\overline{u_0})\|_{L^2(\Omega)}^2+2\int_0^t\int_\Omega \gamma(v)u^2dx\leq \|(-\Delta)^{-\frac12}(u_0-\overline{u_0})\|_{L^2(\Omega)}^2+2 \gamma(v_*)\overline{u_0}^2|\Omega|t.
	\end{equation*}
	On the other hand, noticing that $\overline{w}=\overline{u_0}$, we observe from the Helmholtz equation that
	\begin{align*}
	\|w\|_{H^1(\Omega)}^2=&\int_\Omega (|\nabla w|^2+w^2)dx\non\\
	=&\int_\Omega uwdx\non\\
	=&\int_\Omega (u-\overline{u_0})wdx+\overline{u_0}^2|\Omega|\non\\
	\leq& \|u-\overline{u_0}\|_{H^{-1}(\Omega)}\|w\|_{H^1(\Omega)}+\overline{u_0}^2|\Omega|.
	\end{align*} 
	Thus, by Young's inequality, we obtain that
	\begin{equation*}
	\|w\|_{H^1(\Omega)}^2\leq \|u-\overline{u_0}\|^2_{H^{-1}(\Omega)}+2\overline{u_0}^2|\Omega|,
	\end{equation*}	which completes the proof.
\end{proof}
\begin{remark}
	With the energy estimates in Lemma \ref{est0} and in the same manner as done in \cite[Lemma 3.5]{FJ19}, we may show that the upper bounds of $w$ and hence of $v$ grow at most linearly in time if $n\leq3.$
\end{remark}

\begin{lemma}\label{glm1}
	Assume $n\leq 3$ and $(u,v)$ is a classical solution of system \eqref{chemo1} on $\Omega\times (0,T)$. Then there exists $C(T)>0$ depending on $\Omega$, $T$ and the initial data such that 
	\begin{equation*}
	\sup\limits_{0<t<T}\int_\Omega u(t)\log u(t) dx+\int_0^T\int_\Omega(1+\gamma(v))\frac{|\nabla u|^2}{u}dxds\leq C(T).
	\end{equation*}
\end{lemma}
\begin{proof}
	Multiplying the first equation of \eqref{chemo1} by $\log u$, integrating by parts and applying Young's inequality, we obtain that
	\begin{align}
	\frac{d}{dt}\int_\Omega u\log udx+\int_\Omega \gamma(v)\frac{|\nabla u|^2}{u}dx=&-\int_\Omega \gamma'(v)\nabla v\cdot \nabla udx\non\\
	\leq&\frac12\int_\Omega \gamma(v)\frac{|\nabla u|^2}{u}dx+\int_\Omega 
	\frac{|\gamma'(v)|^2}{\gamma(v)}u|\nabla v|^2dx\non\\
	\leq &\frac12\int_\Omega \gamma(v)\frac{|\nabla u|^2}{u}dx+\int_\Omega \gamma(v)u^2dx+\int_\Omega 
	\frac{|\gamma'(v)|^4}{\gamma(v)^3}|\nabla v|^4dx.\non
	\end{align}
	In view of Lemma \ref{lowbound}, Lemma \ref{keylem1}, Lemma \ref{vbd} and our assumption \eqref{gamma0a} on $\gamma,$ there is $C(T)>0$ depending on the initial data and $\gamma$
	such that for all $(x,t)\in\Omega\times(0,T)$,
	\begin{equation*}
	\frac{|\gamma'(v)|^4}{\gamma(v)^3}(x,t)\leq C(T).
	\end{equation*}
	Therefore, with the aid of the three-dimensional Sobolev embedding
	\begin{equation}\nn
	\|\nabla v\|_{L^4(\Omega)}\leq C\|v\|_{H^2(\Omega)}^{1/2}\|v\|^{1/2}_{L^\infty(\Omega)}+C\|v\|_{L^\infty(\Omega)},
	\end{equation}we infer that \begin{align}
	\int_\Omega\frac{|\gamma'(v)|^4}{\gamma(v)^3}|\nabla v|^4dx\leq& C(T)\int_\Omega|\nabla v|^4dx\non\\
	\leq&C(T)\|v\|^2_{H^2(\Omega)}+C(T).\nn
	\end{align}		
	On the other hand, since $\gamma(v)$ is now bounded from below, we observe from the maximal regularity estimate of heat equations (see \cite{maximal}) and Lemma \ref{est0} that
	\begin{equation}
	\int_0^T\|v\|^2_{H^2(\Omega)}dt
	\leq C \|v_0\|_{H^1(\Omega)}^2+
	C\int_0^T\|u\|_{L^2(\Omega)}^2dt\leq C(T).\nn
	\end{equation}
	Finally, we deduce that
	\begin{equation}
	\int_\Omega u\log udx+\int_0^T\int_\Omega (1+\gamma(v))\frac{|\nabla u|^2}{u}dxdt	\leq C(T),\nn
	\end{equation}
	which completes the proof.
\end{proof}

\subsection{Classical Solution in Dimension Two}
In this part, we deal with the case $n=2$ by a similar argument as done for the classical Keller--Segel models (c.f. \cite{fs2018}). First, we have
\begin{lemma}\label{lemma_LP}
	Assume $n=2$ and let  $(u,v)$ be a classical solution of system \eqref{chemo1} on $\Omega \times (0,T)$. 
	Then there exist  $p \in (1,2)$ and some $C(T) >0$ such that
	\begin{eqnarray*}
		\|u(t)\|_{L^p(\Omega)} \leq C(T)\qquad
		\mbox{for all }t\in(0,T).
	\end{eqnarray*}
\end{lemma}
\begin{proof}
	Multiplying the first equation of \eqref{chemo1} by $ u^{p-1}$ we have
	\begin{eqnarray*}
		\frac{1}{p}\frac{d}{dt} \io u^p \,dx 
		&=& \io u^{p-1} u_t \,dx\\
		&=& \io u^{p-1} \nabla \cdot (\gamma(v) \nabla u + u \gamma'(v) \nabla v)\,dx,
	\end{eqnarray*}
	and by integration by parts, it follows that
	\begin{eqnarray*}
		\frac{1}{p}\frac{d}{dt} \io u^p \,dx +(p-1) \io u^{p-2}\gamma(v)|\nabla u|^2 \,dx
		=
		-(p-1) \io u^{p-1}\gamma'(v) \nabla u \cdot \nabla v \,dx.
	\end{eqnarray*}
	By the Cauchy-Schwarz inequality we have
	\begin{eqnarray*}
		\frac{1}{p}\frac{d}{dt} \io u^p \,dx +\frac{p-1}{2} \io u^{p-2}\gamma(v)|\nabla u|^2 \,dx
		&\leq&
		\frac{p-1}{2} \io \frac{u^{p}|\gamma'(v)|^2}{\gamma(v)} |\nabla v|^2 \,dx \\
		&\leq &
		pM_{\gamma}(T) \io u^{p}|\nabla v|^2 \,dx,
	\end{eqnarray*}
	where we set
	\begin{equation*}
	M_{\gamma}(T)=\sup\limits_{s\in[v_*, v^*(T)]}\frac{|\gamma'(s)|^2}{\gamma(s)} 
	\end{equation*}
	with $v^*(T)\triangleq\frac{ e^{\gamma(v_*)T}\|w_0\|_{L^\infty}+K}{1-\gamma(a)}$ in view of Lemma \ref{keylem1} and Lemma \ref{vbd}.
	Using H\"older's inequality and Young's inequality we obtain that
	\begin{align*}
	\io u^p {|\nabla v|}^2 \,dx
	&\leq
	{\bigg(\io u^{p+1}\,dx \bigg)}^{\frac{p}{p+1}}
	{\bigg(\io {|\nabla v|}^{2(p+1)}\,dx \bigg)}^{\frac{1}{p+1}}\\
	&\leq
	\dfrac{p}{p+1}\io u^{p+1} \,dx
	+
	\dfrac{1}{p+1}\io {|\nabla v|}^{2(p+1)} \,dx,
	\end{align*}
	and in view of Lemma \ref{vbd}, we obtain
	\begin{align*}
	\dfrac{1}{p}\dfrac{d}{dt}\io u^p \,dx 
	+C\io u^{p-2}{|\nabla u|}^2 \,dx
	\leq
	C \io u^{p+1} \,dx
	+ C \io {|\nabla v|}^{2(p+1)}\,dx,
	\end{align*}
	with some $C=C(T)>0$.
	
	On the other hand, by the Sobolev embedding theorem and the regularity theory for heat equations, we deduce that
	\begin{align*}
	\|\nabla v\|_{L^{2(p+1)}(\Omega)}
	\leq
	C \|v\|_{W^{2,\frac{2(p+1)}{p+2}}(\Omega)}
	\leq
	C \|(-\Delta+1) v\|_{L^{\frac{2(p+1)}{p+2}}(\Omega)}
	\end{align*}
	with positive constants $C$. 
By applying the maximal regularity argument \cite{maximal} we estimate that for some fixed $\tau_0\in(0,\frac12 T_{\mathrm{max}})$  and any $t \in (\tau_0,T)$,
\begin{align*}
\int_{\tau_0}^t \io {|\nabla v|}^{2(p+1)}
&\leq 
C \int_{\tau_0}^t  \io 
\|(-\Delta+1) v\|^{2(p+1)}_{L^{\frac{2(p+1)}{p+2}}(\Omega)}\\
&\leq 
C K_{MR}
\bigg(
  \|v(\tau_0)\|^{2(p+1)}_{W^{2,\frac{2(p+1)}{p+2}}(\Omega)}
+
\int_{\tau_0}^{t} 
{\|u\|}^{2(p+1)}_{L^{\frac{2(p+1)}{p+2}}(\Omega)}
\,ds
\bigg)\\
&\leq C \int_{\tau_0}^t \io u^{p+1}+C\|v(\tau_0)\|^{2(p+1)}_{W^{2,\frac{2(p+1)}{p+2}}(\Omega)},
\end{align*}
here we used the relation
\begin{align*}
{\|u\|}^{2(p+1)}_{L^{\frac{2(p+1)}{p+2}}(\Omega)}
\leq
\|u\|^{p+1}_{L^1(\Omega)}\io u^{p+1}.
\end{align*}
Therefore we have that any $t \in (\tau_0,T)$,
\begin{align*}
\dfrac{1}{p}
\io u^p(t) 
+
C\int_{\tau_0}^{t}  \io u^{p-2}{|\nabla u|}^2
\leq&
C \int_{\tau_0}^t \io u^{p+1}
+C\|v(\tau_0)\|^{2(p+1)}_{W^{2,\frac{2(p+1)}{p+2}}(\Omega)}+
\dfrac{1}{p}
\io u^p(\tau_0)\\
\leq& C \int_{\tau_0}^t \io u^{p+1}
+C',
\end{align*}
where $C'>0$ depends only on $\Omega,$ $\|u_0\|_{L^\infty}$ and $\|v_0\|_{W^{1,\infty}(\Omega)}$ due to the local existence result Theorem \ref{local}.

	Finally picking $s>0$ sufficiently large in Lemma \ref{fs} and recalling Lemma \ref{glm1}, we obtain that any $t \in (\tau_0,T)$,
\begin{align*}
\io u^p(\tau) 
\leq C(T),
\end{align*}
which completes the proof together with the local existence result Theorem \ref{local}.
\end{proof}

\noindent\textbf{Proof of Theorem \ref{TH1}.} After the above preparation, we may use the standard bootstrap argument to prove that 
\begin{equation*}
\sup\limits_{0<t<T}\|u(\cdot, t)\|_{L^\infty(\Omega)}\leq C(T)
\end{equation*}for any $T<T_{\mathrm{max}}$
and hence by Theorem \ref{local}, we deduce that $T_{\mathrm{max}}=+\infty$. Therefore, we prove global existence of classical solutions of problem \eqref{chemo1} when $n=2$ if  \eqref{ini}, $\mathrm{(A0)}$ and $\mathrm{(A1)}$ or $\mathrm{(A1')}$ are satisfied. 

Last, in light of the time-independent upper bound of $v$ in Lemma \ref{lm41}, we can proceed along the same lines in \cite{TaoWin17} to show the uniform-in-time boundedness of the classical solutions under assumption $\mathrm{(A2)}$. This completes the proof of Theorem \ref{TH1}.\qed

\begin{remark}\label{rem_tau}
In light of Lemma \ref{unifbddv_tau}, the above discussion still holds true if we replace the second equation of \eqref{chemo1} by 
$$
\tau v_t = \Delta v - v +u
$$
with a constant $\tau>0$.  
\end{remark}

\subsection{Classical Solutions in Dimension Three}
In this part, we study global existence of classical solution when $n=3$. First of all, we show that $\mathrm{(A3)}$ is a stronger condition than $\mathrm{(A2)}$.
\begin{lemma}\label{lemA23}
	A function satisfying $\mathrm{(A0)}$, $\mathrm{(A1)}$ and  $\mathrm{(A3)}$ must fulfill assumption  $\mathrm{(A2)}$ with any $k>1$.
\end{lemma}
\begin{proof}
	First, we point out that under the assumptions $\mathrm{(A0)}$, $\mathrm{(A1)}$ and $\mathrm{(A3)}$, $\gamma'(s)<0$ on $[0,\infty).$ In fact, due to $\mathrm{(A0)}$ and $\mathrm{(A3)}$, we have $\gamma''(s)\geq0$ for all $s\geq0$. Then if there is $s_1\geq0$ such that $\gamma'(s_1)=0$, it must hold that $0=\gamma'(s_1)\leq \gamma'(s)\leq0$ for all $s\geq s_1$ which contradicts to our assumptions $\mathrm{(A0)}$ and $\mathrm{(A1)}$.
	
	Now, we may divide \eqref{gamma3} by $-\gamma(s)\gamma'(s)$ to obtain that
	\begin{equation*}
		-\frac{2\gamma'(s)}{\gamma(s)}\leq -\frac{\gamma''(s)}{\gamma'(s)},\;\;\;\;\forall s>0,
	\end{equation*}
	which indicates that
	\begin{equation*}
		\left(\log(-\gamma^{-2}\gamma')\right)'\leq0.
	\end{equation*}
An integration of above ODI from $v_*$ to $s$ yields that
\begin{equation}
	-\gamma^{-2}(s)\gamma'(s)\leq-\gamma^{-2}(v_*)\gamma'(v_*)\triangleq d>0,
\end{equation}which further implies that
\begin{equation*}
	\left(\frac{1}{\gamma(s)}\right)'\leq d
\end{equation*}
Thus for any $s\geq v_*$, there holds
\begin{equation*}
	\frac{1}{\gamma(s)}\leq d(s-v_*)+\frac{1}{\gamma(v_*)}.
\end{equation*}
As a result, for any $k>1$, we have
\begin{equation}
	\frac{1}{s^{k}\gamma(s)}\leq \frac{d(s-v_*)}{s^k}+\frac{1}{s^k\gamma(v_*)}\rightarrow0,\;\;\;\text{as}\;s\rightarrow+\infty.
\end{equation}This completes the proof.
\end{proof}
\begin{corollary}\label{cor1}
	Assume that $n=3$ and $\gamma(\cdot)$ satisfies $\mathrm{(A0)}$, $\mathrm{(A1)}$ and  $\mathrm{(A3)}$. Then $v$ has a uniform-in-time upper bound in $\Omega\times[0,T_{\mathrm{max}}).$
\end{corollary}

Next, we derive the following energy estimates.
\begin{lemma}\label{lemn3exist}
	Assume $n=3$.  Suppose that  $\gamma(\cdot)$ satisfies  $\mathrm{(A0)}$,  $\mathrm{(A1)}$,  and  $\mathrm{(A3)}$. Then there is $C>0$ depending only on the initial data and $\Omega$ such that
	\begin{equation*}
		\sup\limits_{0\leq t<T_{\mathrm{max}}}\int_\Omega u^2dx\leq C.
	\end{equation*}
\end{lemma}

\begin{proof}
	Multiplying the first equation of \eqref{chemo1} by $2u$ and integrating by parts, we obtain that
	\begin{equation}\label{vk1}
		\frac{d}{dt}\int_\Omega u^2dx+2\int_\Omega \gamma(v)|\nabla u|^2dx=-2\int_\Omega \gamma'(v)u\nabla u\cdot\nabla vdx.
	\end{equation}
On the other hand, we multiply the second equation by $-u^2\gamma'(v)$ to obtain that
\begin{equation*}
	-\int_\Omega v_t u^2\gamma'(v)-2\int_\Omega u\gamma'(v)\nabla u\cdot\nabla v-\int_\Omega u^2\gamma''(v)|\nabla v|^2-\int_\Omega u^2\gamma'(v)v=-\int_\Omega u^3\gamma'(v),
\end{equation*}	
	where we observe that
	\begin{equation*}
		\begin{split}	-\int_\Omega v_t u^2\gamma'(v)=&-\frac{d}{dt}\int_\Omega \gamma(v)u^2+2\int_\Omega uu_t\gamma(v)\\
		=&-\frac{d}{dt}\int_\Omega \gamma(v)u^2+2\int_\Omega u\gamma(v)\Delta(u\gamma(v))\\
		=&-\frac{d}{dt}\int_\Omega \gamma(v)u^2-2\int_\Omega |\nabla(u\gamma(v))|^2.
		\end{split}
	\end{equation*}
Therefore, we have
\begin{equation}
\begin{split}\label{vk3}
		\frac{d}{dt}&\int_\Omega \gamma(v)u^2+2\int_\Omega |\nabla(u\gamma(v))|^2+\int_\Omega u^2\gamma''(v)|\nabla v|^2-\int_\Omega u^3\gamma'(v)\\
		&=-\int_\Omega u\gamma'(v)v-2\int_\Omega u\gamma'(v)\nabla u\cdot\nabla v.
\end{split}
\end{equation}	
Now, multiplying \eqref{vk3} by $\lambda$ with $\lambda>0$ to be specified below and adding the resultant to \eqref{vk1}, we obtain that
\begin{equation}
	\begin{split}
		\frac{d}{dt}\int_\Omega\left(1+\lambda \gamma(v)\right)u^2+&2\lambda\int_\Omega|\nabla(u\gamma(v))|^2+2\int_\Omega\gamma(v)|\nabla u|^2+\lambda\int_\Omega \gamma''(v) u^2|\nabla v|^2\\
		-&\lambda\int_\Omega u^3\gamma'(v)=-\int_\Omega\left(2+2\lambda\right)u\gamma'(v)\nabla u\cdot\nabla v-\lambda\int_\Omega u^2\gamma'(v)v.
	\end{split}
\end{equation}	
Invoking the Young inequality, we infer that
\begin{equation*}
	\begin{split}
		\left|\int_\Omega\left(2+2\lambda\right)u\gamma'(v)\nabla u\cdot\nabla v\right|\leq& 2\int_\Omega  \gamma(v)|\nabla u|^2+\int_\Omega \frac{(1+\lambda)^2|\gamma'(v)|^2}{2\gamma(v)}u^2|\nabla v|^2.	\end{split}
\end{equation*}
Under the assumption 
\begin{equation*}
	2|\gamma'(v)|^2\leq \gamma(v)\gamma''(v),
\end{equation*}
one finds that $\lambda=1$ fulfills 
\begin{equation}\label{condition1}
	\frac{(1+\lambda)^2|\gamma'(v)|^2}{2\gamma(v)}\leq \lambda\gamma''(v).
\end{equation}
As a result, we obtain from above that
\begin{equation}
		\begin{split}
	\frac{d}{dt}\int_\Omega\left(1+ \gamma(v)\right)u^2+2\int_\Omega|\nabla(u\gamma(v))|^2+\int_\Omega u^3|\gamma'(v)|
	\leq-\int_\Omega u^2\gamma'(v)v.
	\end{split}
\end{equation}
Thanks to Corollary \ref{cor1} and Lemma \ref{lowbound}, 
\begin{equation*}
	\left|\int_\Omega u^2\gamma'(v)v\right|\leq C\int_\Omega u^2dx.
\end{equation*}
Thus, we obtain that
\begin{equation}
\begin{split}
\frac{d}{dt}\int_\Omega\left(1+ \gamma(v)\right)u^2+2\int_\Omega|\nabla(u\gamma(v))|^2+\int_\Omega u^3|\gamma'(v)|\leq C\int_\Omega u^2dx.
\end{split}
\end{equation}
On the other hand, since now $v$ is bounded from above and below, there is $\gamma_*>0$ such that $\gamma_*\leq\gamma(v)\leq \gamma(v_*)$ and it follows from \eqref{ubw2} that
\begin{equation}
\int_t^{t+\tau} \int_\Omega(1+\gamma(v))u^2dxds\leq C.
\end{equation}
Now we may apply the uniform Gronwall inequality together with the local existence result to conclude that
\begin{equation*}
	\int_\Omega (1+\gamma(v))u^2dx\leq C,\;\;\forall\;t\in[0,T_{\mathrm{max}}).
\end{equation*}This completes the proof.
\end{proof}
\begin{remark}
	Our assumption $\mathrm{(A3)}$ is independent of the coefficients of the system. If we replace the second equation of system \eqref{chemo1} by $v_t-\alpha\Delta v+\beta v=\theta u$ with some $\alpha,\beta,\theta>0$, one easily checks that condition \eqref{condition1} becomes
	\begin{equation}
			\frac{(1+\alpha\lambda)^2|\gamma'(v)|^2}{2\gamma(v)}\leq \alpha\lambda\gamma''(v),
	\end{equation}
which holds with $\lambda=1/\alpha$ under assumption $\mathrm{(A3)}$.
\end{remark}
\noindent\textbf{Proof of Theorem \ref{TH3d}.} With the aid of Lemma \ref{lemn3exist}, we may further use standard the bootstrap argument to prove that 
\begin{equation*}
\sup\limits_{0<t<T}\|u(\cdot, t)\|_{L^\infty(\Omega)}\leq C
\end{equation*}for any $T<T_{\mathrm{max}}$. Since similar argument is given in detail in \cite{Anh19}, we omit the proof here.
Finally, by Theorem \ref{local}, we deduce that $T_{\mathrm{max}}=+\infty$ and Theorem \ref{TH3d} is proved.\qed


\section{The Critical Mass Phenomenon with $\gamma(v)=e^{-v}$}

This section is devoted to the special case $\gamma(v)=e^{-v}$. Namely, we consider the following initial Neumann boundary value problem:
\begin{equation}
\begin{cases}\label{chemo2}
u_t=\Delta (ue^{-v})&x\in\Omega,\;t>0,\\
v_t-\Delta v+v=u&x\in\Omega,\;t>0,\\
\partial_\nu u=\partial_\nu v=0,\qquad &x\in\partial\Omega,\;t>0,\\
u(x,0)=u_0(x),\;v(x,0)=v_0(x)\qquad & x\in\Omega,
\end{cases}
\end{equation}with $\Omega\subset\mathbb{R}^2.$ 
\subsection{Uniform-in-time Boundedness with Sub-critical Mass}
In this part, we first prove the following uniform-in-time boundedness of the classical solutions with sub-critical mass.
\begin{proposition}\label{prop1}
	Assume $n=2$ and let
	\begin{equation*}
	\Lambda_c=\begin{cases}
	8\pi\qquad\text{if}\;\Omega=\{x\in\mathbb{R}^2;\;|x|<R\}\;\;\text{and}\;(u_0,v_0) \;\text{is radial in}\;x,\\
	4\pi\qquad\text{otherwise.}		
	\end{cases}
	\end{equation*}
	If $\Lambda\triangleq\int_\Omega u_0dx<\Lambda_c$, then the global classical solution $(u,v)$ to system \eqref{chemo2} is uniformly-in-time bounded in the sense that
	\begin{equation*}
	\sup\limits_{t\in(0,\infty)}\left(\|u(\cdot,t)\|_{L^\infty(\Omega)}+\|v(\cdot,t)\|_{L^\infty(\Omega)}\right)<\infty.
	\end{equation*}
\end{proposition}
First, system \eqref{chemo2} is a dissipative dynamical system.
\begin{lemma} There holds
	\begin{equation}\label{Lyapunov}
	\frac{d}{dt}\mathcal{F}(u,v)(t)+	\int_\Omega ue^{-v}\left|\nabla \log u-\nabla v\right|^2dx+\|v_t\|_{L^2(\Omega)}^2=0,
	\end{equation} where the  functional $\mathcal{F}(\cdot,\cdot)$ is defined by
	\begin{equation*}\non
	\mathcal{F}(u,v)=\int_\Omega \left(u\log u+\frac12|\nabla v|^2+\frac12 v^2-uv\right)dx.
	\end{equation*}
\end{lemma}
\begin{proof}
	Multiplying the first equation of \eqref{chemo2} by $\log u-v$, the second equation of \eqref{chemo2} by $v_t$ and integrating by parts, then adding the resultants together, we get
	\begin{equation}\non
	\frac{d}{dt}\int_\Omega \left(u\log u+\frac12|\nabla v|^2+\frac12 v^2-uv\right)dx+\int_\Omega ue^{-v}\left|\nabla \log u-\nabla v\right|^2dx+\|v_t\|_{L^2(\Omega)}^2=0.
	\end{equation}This completes the proof.	
\end{proof}
Since the energy $\mathcal{F}(\cdot,\cdot)$ is the same as that of the classical Keller--Segel model, we may recall \cite[Lemma 3.4]{Nagai97} stated as follows.
\begin{lemma}\label{ublm1} 
	If $\Lambda<\Lambda_c$, there exists a positive constant $C$ independent of $t$ such that
	\begin{equation*}
	\|v\|_{H^1(\Omega)}\leq C,\;\;\int_\Omega uvdx\leq C\qquad\text{and}\qquad|\mathcal{F}(u(t),v(t))|\leq C,\;\;\forall \;t\geq0.
	\end{equation*}	
\end{lemma}
Next, we aim to derive a time-independent upper bound of $v$ with subcritical mass. For this purpose, we need the following uniform-in-time estimates.
\begin{lemma}\label{ublm2}If $\Lambda<\Lambda_c$, then there holds
	\begin{equation*}
	\sup\limits_{t\geq0}\int_t^{t+1}\int_\Omega  e^{-v(s)}u^2(s)dxds\leq C,
	\end{equation*}where $C>0$ depends on $\Omega$ and the initial data only.
\end{lemma}
\begin{proof}
	Multiplying the first equation of \eqref{chemo2} by $w$ and integrating over $\Omega$, we obtain that
	\begin{equation*}
	\int_\Omega u_t wdx=\int_\Omega e^{-v}u\Delta w dx.
	\end{equation*}
	Recalling that $w-\Delta w=u$, the above equality implies that
	\begin{align*}
	\int_\Omega (-\Delta w_t+w_t)w dx+\int_\Omega e^{-v}u^2dx=\int_\Omega e^{-v}uwdx.
	\end{align*}
	Hence, we have
	\begin{equation}\label{unb0}
	\frac12\frac{d}{dt}(\|\nabla w\|_{L^2(\Omega)}^2+\|w\|_{L^2(\Omega)}^2)+\int_\Omega e^{-v}u^2dx=\int_\Omega e^{-v}uwdx\leq \frac12\int_\Omega e^{-v}u^2dx+\frac12\int_\Omega e^{-v}w^2dx.
	\end{equation}
In view of Lemma \ref{lm2}, we observe that
\begin{equation}
	\int_\Omega e^{-v}w^2dx\leq \int_\Omega w^2dx\leq C\|u\|^2_{L^1(\Omega)}=C\Lambda^2.
\end{equation}
On the other hand, by integration by parts and Young's inequality, we infer that
\begin{align*}
\|\nabla w\|_{L^2(\Omega)}^2+\|w\|_{L^2(\Omega)}^2=&\int_\Omega w udx\non\\
\leq&\int_\Omega e^{-v}u^2dx+\int_\Omega e^{v}w^2dx.
\end{align*}
Thanks to H\"older's inequality  and Lemma \ref{lm2}, we infer that
\begin{equation*}
	\int_\Omega e^{v}w^2dx\leq\left(\int_\Omega e^{2v}dx\right)^{1/2}\left(\int_\Omega w^4dx\right)^{1/2}\leq C
\end{equation*}
with $C>0$ depending only on the initial data and $\Omega$, where we also used the 2D Trudinger-Moser inequality \cite[Theorem 2.2]{Nagai97} to infer that \begin{align*}\non
\int_\Omega e^{2v}dx\leq Ce^{C(\|\nabla v\|_{L^2(\Omega)}^2+\|v\|_{L^2(\Omega)}^2)}
\end{align*}with $C>0$ depending only on $\Omega.$
Therefore, we deduce from above that
\begin{equation}\label{unb0a}
		\frac{d}{dt}(\|\nabla w\|_{L^2(\Omega)}^2+\|w\|_{L^2(\Omega)}^2)+\frac12\int_\Omega e^{-v}u^2dx+\frac12(\|\nabla w\|_{L^2(\Omega)}^2+\|w\|_{L^2(\Omega)}^2)\leq C.
\end{equation}
Then we may apply the ODE technique to conclude that
\begin{equation*}
	\sup\limits_{t\geq0}(\|\nabla w\|_{L^2(\Omega)}^2+\|w\|_{L^2(\Omega)}^2)\leq C
\end{equation*}with $C>0$ depending only on the initial data and $\Omega$.
Moreover,   an integration of \eqref{unb0a} with respect to time from $t$ to $t+1$ together with the fact $\sup\limits_{t\geq0}\|w\|_{H^1}\leq C$ will finally yields to our assertion. This completes the proof.
\end{proof}

\begin{remark}\label{remv}If $w$ or $v$ has a uniform-in-time upper bound, then one  has
\begin{equation*}
\sup\limits_{t\geq0}\left(\|v\|_{H^1(\Omega)}+\int_\Omega uv dx+|\mathcal{F}(u(t),v(t))|+\int_t^{t+1}\int_\Omega  e^{-v(s)}u^2(s)dxds\right)\leq C,
\end{equation*}where $C>0$ depends on $\Omega$ and the initial data only.
\end{remark}
\begin{proof}
	If $w$ is uniformly-in-time bounded, then it follows from Lemma \ref{vbd} that $\sup\limits_{t\geq0}\|v(t,\cdot)\|_{L^\infty(\Omega)}\leq C$ by some $C>0$ independent of $t$. As a result, we infer that
	\begin{equation}
		\begin{split}
		\int_\Omega \left(u\log u+\frac{1}{2}|\nabla v|^2+\frac12v^2 \right)dx=&\mathcal{F}(u,v)+\int_\Omega u vdx\\
		\leq&\mathcal{F}(u,v)+\|v\|_{L^\infty(\Omega)}\int_\Omega udx\\
		\leq&\mathcal{F}(u_0,v_0)+C\Lambda,
		\end{split}
	\end{equation}which indicates that
	\begin{equation}
		\sup\limits_{t\geq0}\left(\|v\|_{H^1(\Omega)}+\int_\Omega uv dx+|\mathcal{F}(u(t),v(t))|\right)\leq C.
	\end{equation}
Then we may concludes the proof in the same manner as in Lemma \ref{ublm2}.
\end{proof}

\begin{lemma}\label{ublm3}
	If $\Lambda<\Lambda_c$, then there exists $C>0$ depending on $\Omega$ and the initial data such that for all $x\in\Omega$
	\begin{equation*}
	\sup\limits_{t\geq0}	v(x,t)\leq C.
	\end{equation*}	
\end{lemma}

\begin{proof}
	First, we apply the Sobolev embedding theorem, the elliptic regularity theorem and H\"older's inequality to infer that
	\begin{align*}
	\|w\|_{L^\infty(\Omega)}\leq &C\|w\|_{W^{2,\frac32}(\Omega)}\non\\
	\leq&C\|u\|_{L^{\frac32}(\Omega)}\non\\
	=&C\left(\int_\Omega u^{\frac32}dx\right)^{\frac23}\non\\
	\leq&C\left(\int_\Omega u^2 e^{-v}dx\right)^{\frac12}\left(\int_\Omega e^{3v}dx\right)^{\frac16}\non\\
	\leq&C\left(\int_\Omega u^2 e^{-v}dx\right)^{1/2},
	\end{align*}
	where we used the 2D Trudinger-Moser inequality \cite[Theorem 2.2]{Nagai97} to deduce that
	\begin{align*}\non
	\int_\Omega e^{3v}dx\leq Ce^{C(\|\nabla v\|_{L^2(\Omega)}^2+\|v\|_{L^2(\Omega)}^2)}
	\end{align*}with $C>0$ depending only on $\Omega.$
	Thus, by Lemma \ref{ublm2}, for any $t\geq0$, there holds
	\begin{equation*}
	\int_t^{t+1}\|w\|_{L^\infty(\Omega)}^2ds\leq C\int_{t}^{t+1}\int_\Omega u^2 e^{-v}dxds\leq C,
	\end{equation*}which due to Young's inequality indicates that
	\begin{equation*}\non
	\int_t^{t+1}\|w\|_{L^\infty(\Omega)}ds\leq\int_t^{t+1}\|w\|^2_{L^\infty(\Omega)}+C\leq C.
	\end{equation*}
	Hence, for any $x\in\Omega$ and $t\geq0$, we obtain that
	\begin{equation}\label{unint00}
	\int_t^{t+1}w(x,s)ds\leq\int_t^{t+1}\|w\|_{L^\infty(\Omega)}ds\leq C.
	\end{equation}
	Observing that 
	\begin{equation}\non
	w_t+ue^{-v}=(I-\Delta)^{-1}[ue^{-v}]\leq (I-\Delta)^{-1}[u]=w,
	\end{equation}
	we may fix $x\in\Omega$ and apply the uniform Gronwall inequality Lemma \ref{uniformGronwall} to deduce that
	\begin{equation}\non
	w(x,t)\leq C\;\;\text{for all}\;t\geq1.
	\end{equation}
	Since  $C>0$ above is independent of $x$ and
	\begin{equation}\non
	w(x,t)\leq w_0(x)e^{e^{-v_*}}\leq ew_0(x)\quad\text{for any}\;x\in\Omega\;\text{and}\;t\in[0,1]
	\end{equation}
	due to Lemma \ref{keylem1}, we conclude that
	\begin{equation}\non
	\sup\limits_{t\geq0} w(x,t)\leq C.
	\end{equation}
	As a result, $v$ is uniformly-in-time bounded as well according to Lemma \ref{vbd}.
	This completes the proof.
\end{proof} 
\noindent\textbf{Proof of Proposition \ref{prop1}.} 
 Proceeding along the same lines in \cite{TaoWin17}, we can invoke  the time-independent upper bound of $v$ to show the uniform-in-time boundedness of the classical solutions, which concludes the proof.
\qed
\begin{remark}
In view of Remark \ref{remv}, if $u$ blows up at time infinity, then $\|v\|_{L^\infty(\Omega)}$ and $\|w\|_{L^\infty(\Omega)}$ cannot be uniformly-in-time bounded and thus we have
\begin{equation}
\limsup\limits_{t\nearrow +\infty}\|(I-\Delta)^{-1}[u](\cdot,t)\|_{L^\infty(\Omega)}=\limsup\limits_{t\nearrow +\infty}\|v(\cdot,t)\|_{L^\infty(\Omega)}=+\infty.
\end{equation}
\end{remark}

\subsection{Unboundedness with Super-ciritical Mass}\label{section_unboundedness}
In this part we construct blowup solutions in infinite time. 
Since the system \eqref{chemo2} has the similar energy structure and the same stationary problem as the Keller-Segel system, 
we may verify existence of blowup solutions following the idea in \cite{ssMAA2001,HW01}. 

Stationary solutions $(u,v)$ to \eqref{chemo2} satisfy that
\begin{align*}
\begin{cases}
0=\nabla \cdot ue^{-v} \nabla \left(\log u - v \right)
&\mathrm{in}\ \Omega, \\[1mm]
0=\Delta v -v+u
&\mathrm{in}\ \Omega, \\
u > 0, \ v > 0& \mathrm{in}\ \Omega, \\ 
\displaystyle \frac{\partial u}{\partial \nu} =\frac{\partial v}{\partial \nu}  =0 & \mathrm{on}\ \partial \Omega.
\end{cases}
\end{align*}
Put $\Lambda = \|u\|_{L^1(\Omega)} \in (0,\infty)$. 
In view of the mass conservation and the boundary condition, the set of equilibria consists of solution to the following problem:
\begin{align} \label{eqn:biharmoniceqn}
\begin{cases}
\displaystyle	v-\Delta v = \frac{\Lambda}{\int_\Omega e^{v}}
e^{ v} & \mathrm{in}\ \Omega, \\[1mm]
\displaystyle u = \frac{\Lambda}{\int_\Omega e^{ v}}e^{ v}
& \mathrm{in}\ \Omega,  \\[1mm]
\displaystyle \frac{\partial v}{\partial \nu}  =0 & \mathrm{on}\ \partial \Omega. 
\end{cases}
\end{align}
Proceeding the same way as in \cite[Lemma 3.1]{win_JDE}, 
we have the following result.
\begin{proposition}
	\label{prop:conv-solu-subseq}
	Let $(u,v)$ be a classical non-negative  solution to \eqref{chemo2} in
	$\Omega \times (0,\infty)$. If the solution is uniformly-in-time bounded, 
	there exist a sequence of time $\{t_k\} \subset
	(0,\infty)$ and a  solution $(u_s,v_s)$ to \eqref{eqn:biharmoniceqn} such that 
	$\lim_{k \rightarrow \infty} t_k = \infty$ and that 
	\[
	\lim_{k \rightarrow \infty} (u(t_k),v(t_k)) = (u_s,v_s) \quad \mbox{{\rm in} }C^2(\overline{\Omega}).
	\]
	as well as
	$$
	\mathcal{F}(u_s,v_s) \leq \mathcal{F}(u_0,v_0).
	$$
\end{proposition}
\begin{remark}
Observe that for any solution $(u_s,v_s)$ to \eqref{eqn:biharmoniceqn}, $u_s$ is strictly positive on $\overline{\Omega}$ (see, e.g., \cite[Sect. 2]{FLP07}).  Assume for any $j\geq1$, there is $t_j>0$ and $x_j\in\Omega$ such that $u(t_j,x_j)<1/j$. Then by a similar compactness argument as in  \cite[Lemma 3.1]{win_JDE}, one may extract a time subsequence, still denoted by $t_j$, such that $u(t_j)$ converges to some $u_s$ in $C^2(\overline{\Omega})$, which leads to a contradiction since $u_s$ is strictly positive. Thus, we infer that for any uniformly-in-time bounded solution $(u,v)$, $u$ is strictly positive for $(t_0,+\infty)\times\overline{\Omega}$ with some sufficiently large $t_0$ and we can now apply the non-smooth Lojasiewicz--Simon inequality established in \cite{FLP07} (see, also \cite{JZ09,J18}) to deduce that  
	\[
\lim_{t \rightarrow+ \infty} (u(t),v(t)) = (u_s,v_s) \quad \mbox{{\rm in} }C^2(\overline{\Omega}).
\]
\end{remark}

For $\Lambda >0$ put
\[
\mathcal{S}(\Lambda) \triangleq \left\{ (u,v) \in C^2(\overline{\Omega}) :
(u,v ) \mbox{ is a solution to \eqref{eqn:biharmoniceqn}
	 } \right\}.
\]
Here we recall the quantization property of solutions to  \eqref{eqn:biharmoniceqn}.  
By \cite[Theorem 1]{ssAMSA2000} for $\Lambda \not\in 4\pi \mathbb{N}$  there exists some $C>0$ such that
\[
\sup \{ \|(u,v)\|_{L^\infty (\Omega)} : (u,v) \in
\mathcal{S}(\Lambda) \} \leq C
\]
and
\[
F_\ast(\Lambda) := \inf \{ \mathcal{F}(u,v) : (u,v) \in
\mathcal{S}(\Lambda) \} \geq - C.
\] 
Thus by taking account of Lemma \ref{prop:conv-solu-subseq}, 
for a pair of functions $(u_0,v_0)$ satisfying 
\begin{align*}
\begin{cases}
\|u_0\|_{L^1(\Omega)} = \Lambda\not\in 4\pi \mathbb{N},\\[2pt]
\mathcal{F}(u_0,v_0) <F_\ast(\Lambda),
\end{cases}
\end{align*}
 the corresponding global solution must blow up in infinite time. 

From now on we will construct an example satisfying the above condition based on calculations in \cite{fs5}. 
A straightforward calculation leads us to the following lemma. 
\begin{lemma}\label{lemm:sol-fund-sys}
	For any $\lambda >0$ the following functions
	\[
	u_\lambda (x) := 
	\frac{8\lambda^2}{(1+\lambda^2|x|^2)^2}, \quad 
	v_\lambda (x) := 
	2 \log \frac{\lambda}{1+\lambda^2|x|^2} +
	\log 8\qquad \mbox{for all } x \in \R^2, 
	\]
	satisfy 
	\[
	e^{v_\lambda} = u_\lambda,\quad 0=\Delta v_\lambda+u_\lambda, \quad
	\int_{\R^2} u_\lambda = 8 \pi. 
	\]
\end{lemma}

We modify the above functions as: 
for any $\lambda \geq 1$ and $r \in (0,1)$,
\[
\overline{u}_\lambda (x) := 
\frac{8\lambda^2}{(1+\lambda^2|x|^2)^2}, \quad 
\overline{v}_{\lambda, r} (x) := 
2 \log \frac{1+\lambda^2r^2}{1+\lambda^2|x|^2} +
\log 8,
\]
and by simple calculations it follows that
\begin{eqnarray*}
	\overline{u}_\lambda (x) \leq 8\lambda^2, \quad  
	\overline{v}_{\lambda, r} (x) >  \log 8>0
	\quad \mbox{in }  B(0,r).
\end{eqnarray*}

\begin{proof}[Proof of Theorem \ref{TH4}]
	Let $\Lambda \in (8\pi, \infty) \setminus 4\pi \mathbb{N}$. 
	Take $r\in (0,1)$ and $q \in \Omega$ such that $B(q,2r) \subset \Omega$.  By translation, we may assume that $q=0$. For any $r_1 \in
	(0,r)$, let $\phi_{r,r_1}$ be a smooth and radially symmetric function satisfying
	\[
	\phi_{r,r_1}(B(0,r_1)) =1, \ 
	0\leq \phi_{r,r_1} \leq 1, \ 
	\phi_{r,r_1}(\R^2 \setminus
	B(0,r)  ) =0,\ x \cdot \nabla \phi_{r,r_1}(x) \leq 0.
	\]
	Noting that $$f(\lambda):= 1 - \frac{1}{1+(\lambda
		r_1)^2} \to 1 \quad \mbox{ as } \lambda \to \infty,$$
	and that
	$$
	f'(\lambda)=  \frac{2\lambda r_1}{(1+(\lambda r_1)^2)^2}>0 \quad \mbox{for }\lambda>0,
	$$ 
	we have that $1>f(\lambda) \geq f(1)$ for all $\lambda \geq 1$.

	Now we define the pair $(u_0,v_0)\triangleq(a\overline{u}_\lambda \phi_{r,r_1},
	a\overline{v}_{\lambda, r} \phi_{r,r_1})$
	with some $a > \Lambda /8\pi > 1$.  Then, we prove that
	\begin{lemma}
		There  is a sufficiently large $\lambda>1$ and 
		$a> \Lambda /8\pi$  such that
		\begin{eqnarray}\label{mass_inequality}
		&& \io u_0 = \Lambda.
		\end{eqnarray}
	\end{lemma}
\begin{proof}
	Firstly by changing variables,  we see that
	\begin{eqnarray*} 
		\int_{B(0,\ell)} \overline{u}_\lambda \nonumber 
		&=& 8 \int_{B(0,\ell)}
		\frac{\lambda^2}{(1+\lambda^2|x|^2)^2}\,dx \nonumber \\
		&=&8 \int_{B(0,\lambda \ell)} \frac{dy}{(1+|y|^2)^2}
		\nonumber \\
		&=&16 \pi \int_0^{\lambda \ell} \frac{s}{(1+s^2)^2}\,ds
		\nonumber \\
		&=&8 \pi \int_0^{(\lambda \ell)^2}
		\frac{d\tau}{(1+\tau)^2} \nonumber \\
		&=& 8 \pi \cdot 
		\left( 1 -\frac{1}{1+(\lambda\ell)^2} \right)
		\quad \mbox{ for } \ell >0,
	\end{eqnarray*}
	and that
	\begin{equation}\label{eqn:relat-a-lambda}
	8 \pi \cdot 
	\left( 1 -\frac{1}{1+(\lambda r_1)^2} \right)
	<
	\io \overline{u}_\lambda \phi_{r,r_1}
	<
	8 \pi \cdot 
	\left( 1 -\frac{1}{1+(\lambda r)^2} \right).						 
	\end{equation}
	Then there is a unique constant $a=a(r_1,r,\lambda)$ satisfying
	\begin{equation}\label{bound_of_a}
	 \frac{\Lambda }{8\pi} 
	 \leq  a 
	\leq 	\frac{\Lambda}{8\pi f(1)}
	\end{equation}
	and (\ref{mass_inequality}). 
\end{proof}
Next, we want to show that $\mathcal{F}(u_0,v_0)$ can be sufficiently negative as $\lambda\rightarrow+\infty.$ First, we note that
\begin{lemma}There is $C>0$ such that
\begin{eqnarray}
\int_\Omega u_0 \log u_0 & \leq & 
16a\pi  \cdot \log \lambda + C \ \mbox{ as }
\ \lambda \rightarrow \infty.
\end{eqnarray}
\end{lemma}
\begin{proof}
Observe  that  
	\begin{eqnarray*}
		\int_\Omega u_0 \log u_0 
		& \leq & a \io
		\overline{u}_\lambda \log \overline{u}_\lambda + a \log a
		\io \overline{u}_\lambda.
	\end{eqnarray*}
	Since $\log \overline{u}_\lambda \leq \log (8\lambda^2) = 2 \log \lambda +\log 8$ and $\io \overline{u}_\lambda \leq 8\pi$,  
	\begin{eqnarray}\label{lyapest1}
	\int_\Omega u_0 \log u_0 & \leq & 
	2a \cdot 8\pi \cdot \log \lambda + C \ \mbox{ as }
	\ \lambda \rightarrow \infty,   
	\end{eqnarray}
	where we remark that the constant $C$ is independent of $a$ in view of \eqref{bound_of_a}.
\end{proof}	
\begin{lemma}
	There exists $C>0$ such that
	\begin{equation}\label{lyapest2}
		\int_\Omega u_0 v_0 dx\geq32a^2\pi\log \lambda-C
		 \ \mbox{ as }
\ \lambda \rightarrow \infty,
	\end{equation}as well as
	\begin{equation}\label{lyapest3}
	\frac12\int_\Omega\left(v_0^2+|\nabla v_0|^2\right)dx\leq 16a^2\pi\log \lambda+C
	 \ \mbox{ as }
\ \lambda \rightarrow \infty.
	\end{equation}
\end{lemma}	
\begin{proof}
Using $\overline{v}_{\lambda ,r}>0$ in $B(0,r)$, we
	see that   
	\begin{eqnarray*}
		\int_\Omega u_0 v_0 & \geq & a^2 \int_{B(0,r_1)}
		\overline{u}_\lambda \overline{v}_{\lambda ,r}.
	\end{eqnarray*}
	Since 
	$$\overline{v}_{\lambda, r} (x) >  
	2 \log \frac{1+\lambda^2 r^2}{1+\lambda^2 |x|^2} \quad
	\mbox{ for } x \in B(0,r_1), 
	$$
	then we have that 
	\begin{eqnarray*}
		\int_\Omega u_0 v_0 & \geq & a^2 \int_{B(0,r_1)}
		\overline{u}_\lambda  \cdot 2 \log \frac{1+\lambda^2
			r^2}{1+\lambda^2 |x|^2}   \\
		& > & 4a^2 \log (\lambda r) \int_{B(0,r_1)} \overline{u}_\lambda 
		- 2a^2 \int_{B(0,r_1)}
		\overline{u}_\lambda \log (1+\lambda^2 |x|^2)  
	\end{eqnarray*}
	and that
	\begin{eqnarray*}
		\int_{B(0,r_1)}
		\overline{u}_\lambda \log (1+\lambda^2 |x|^2) & = &
		8 \int_{B(0,r_1)}
		\frac{\lambda^2 \log (1+\lambda^2 |x|^2)}{(1+\lambda^2 |x|^2)^2}\,dx  \\
		& = &  16\pi  \int_0^{\lambda r_1}
		\frac{ s\log (1+s^2)}{(1+s^2)^2}\,ds  \\
		& < &  8\pi  \int_0^{\infty}
		\frac{ \log (1+\xi)}{(1+\xi)^2}\,d\xi < \infty.
	\end{eqnarray*}
	Combining these with (\ref{eqn:relat-a-lambda}), we obtain that
	\begin{eqnarray*}
	\nn
	\int_\Omega u_0 v_0 
	&\geq& 4a^2 \log (\lambda r)\cdot 8\pi
	\left( 1 -\frac{1}{1+(\lambda r_1)^2} \right)- C \\
	&\geq& 32 \pi a^2 \log \lambda - C'
	\end{eqnarray*}
	for $\lambda > 1$, $r \in (0,1)$ and $r_1 \in (0,r)$ with some positive constants $C, C'$. 
	We remark that the constant $C'$ is independent of $a$ due to \eqref{bound_of_a}.
	
	On the other hand, since
	$$
	\dfrac{1+\lambda^2 r^2}{1+\lambda^2 |x|^2} \leq 
	\left(\dfrac{1+\lambda r}{\lambda |x|} \right)^2,
	$$ 
	we see that for $\lambda \geq 1$
	\[
	|\overline{v}_{\lambda, r} (x)|
	\leq 4 \log \frac{1+r}{|x|} + \log 8 
	\ \mbox{ in } B(0,r).
	\]
	Hence it follows from straightforward calculations that there is a positive constant $C$ satisfying 
	\begin{eqnarray*}
	\nn
	\frac{1}{2}\io v_0^2 
	&\leq& 
	a^2 \int_{B(0,r)} \left(4 \log \frac{1+r}{|x|} + \log 8 \right)^2\\
	&\leq& C,
	\end{eqnarray*}
	where the constant $C$ is independent of $a$ due to \eqref{bound_of_a}.
	Moreover by the direct calculations,
	\[
	|\nabla \overline{v}_{\lambda, r} (x)|
	= \dfrac{4 \lambda^2 |x|}{1+\lambda^2 |x|^2} 
	\ \mbox{ in } B(0,r) 
	\]
	and that
	\begin{eqnarray*}
		\io |\nabla v_0|^2 
		&\leq&
		16a^2 \int_{B(0,r)} \dfrac{ \lambda^4 |x|^2}{(1+\lambda^2 |x|^2)^2}\,dx\\
		&=&16a^2 \int_{B(0,\lambda r)} \dfrac{  |y|^2}{(1+|y|^2)^2}\,dy\\
		&=&32 \pi a^2 \int_0^{\lambda r} \dfrac{  s\cdot s^2}{(1+s^2)^2}\,ds\\
		&=&16 \pi a^2 \int_0^{(\lambda r)^2} \dfrac{ \tau}{(1+\tau)^2}\,d\tau\\
		&\leq&16 \pi a^2\int_0^{(\lambda r)^2} \dfrac{ 1}{1+\tau}\,d\tau\\
		&=&16 \pi a^2\cdot \log (1+ (\lambda r)^2)
	\end{eqnarray*}
	thus
	\begin{eqnarray*}
	\frac{1}{2}\io |\nabla v_0|^2 
	\leq 16 \pi a^2 \log \lambda +C'',
	\end{eqnarray*}
	where we again remark that the constant $C''$ is independent of $a$ due to \eqref{bound_of_a}.
\end{proof}	
	
	Collecting \eqref{lyapest1}, \eqref{lyapest2} and \eqref{lyapest3}, we infer that for $r \in (0,1)$ and $r_1 \in (0,r)$ there exists some $C=C(r, r_1, \phi_{r,r_1})$ such that
	\begin{eqnarray}\label{lyapunov_decreasing}
	\nn
	\mathcal{F}(u_0,v_0)
	& \leq & 
	16 \pi a \log \lambda 
	- 32 \pi a^2 \log \lambda
	+16 \pi a^2 \log \lambda
	+C \\
	\nn
	& = & -16 \pi a (a-1)\log \lambda +C\\
	& \leq & - 2\Lambda  \left(\frac{\Lambda}{8\pi} -1\right)
	\log \lambda + C  \rightarrow
	-\infty \ \mbox{ as } \ \lambda \rightarrow \infty,
	\end{eqnarray}
	where we recalled that \eqref{bound_of_a} implies
	$$a(a-1)> 
	\frac{\Lambda }{8\pi} \left(\frac{\Lambda}{8\pi} -1\right).$$
	
	In the last step,  we construct a suitable initial data based on the above discussion. 
	For $\Lambda \in (8\pi, \infty) \setminus  4\pi \mathbb{N}$, we first fix $0<r_1 < r$ 
	and function $\phi_{r,r_1}$. 
	Secondly in view of \eqref{lyapunov_decreasing} 
	we can choose some $\lambda > 1$ such that 
	$$
	- 2\Lambda  \left(\frac{\Lambda}{8\pi} -1\right)
	\log \lambda + C < F_* (\Lambda),
	$$
	where $C=C(r, r_1, \phi_{r,r_1})$  is the constant in \eqref{lyapunov_decreasing}.  
	Finally we choose $a$ satisfying \eqref{mass_inequality} and \eqref{bound_of_a}.
	Therefore by the above discussion $(u_0,v_0)$ also satisfies 
	\begin{equation}
		\mathcal{F}(u_0,v_0)<F_*(\Lambda).
	\end{equation}
	Thus let $(u,v)$
	be the solution to \eqref{chemo2} with the initial function
	$(u_0,v_0)$. If the solution is globally bounded in time,
	Proposition \ref{prop:conv-solu-subseq} guarantees that there are a
	subsequence $\{t_k\} \subset (0,\infty)$ and a stationary solution
	$(u_s,v_s)$ satisfying that
	\[
	\lim_{t_k \rightarrow \infty} (u(t_k),v(t_k)) = (u_s,
	v_s) \ \mbox{ in } C^1(\overline{\Omega})
	\]
	and that
	\[
	\mathcal{F}(u_s,v_s) < F_\ast(\Lambda).
	\]
	It contradicts to the definition of $F_\ast(\Lambda)$. 
	Thus the proof is complete. 
\end{proof}
\bigskip
\bigskip

\noindent\textbf{Acknowledgments} \\

K. Fujie is supported by Japan Society for the Promotion of Science (Grant-in-Aid for Early-Career Scientists; No.\ 19K14576).


\begin{thebibliography}{99}
	\itemsep=0pt

\bibitem{Anh19} J. Ahn and C. Yoon,
Global well-posedness and stability of constant equilibria in parabolic-elliptic chemotaxis systems without gradient sensing,
Nonlinearity, \text{32} (2019), 1327--1351.

\bibitem{BBTW15}
N. Bellomo, A. Belouquid, Y. Tao and M. Winkler,
Toward a mathematical theory of Keller--Segel models of pattern formation in biology tissues,
Math. Mod. Meth. Appl. Sci., \textbf{25} (2015), 1663--1763.

\bibitem{BCM10}
A. Blanchet, J.A. Carrillo and N. Masmoudi,
Infinite time aggregation for the critical Patlak-Keller-Segel model in $\mathbb{R}^2$,
Commun. Pure Appl. Math., \textbf{61} (2008), 1449--1481.




\bibitem{BS}
H. Br\'ezis and W. Strauss,
Semi-linear second-order elliptic equations in $L^1$,
J.\ Math.\ Soc.\ Japan, {\bf 25} (1973), 565--590.
%

\bibitem{Cao}
 X. Cao,
  Global bounded solutions of the higher-dimensional Keller-Segel system under smallness conditions in optimal spaces, 
  Discrete Contin. Dynam. Syst. Ser. A, \textbf{35}(2015), 1891--1904.

\bibitem{CS}
T. Cie\'slak and C. Stinner,
New critical exponents in a fully parabolic quasilinear Keller-Segel system and applications to volume filling models, 
J. Differential Equations, \textbf{258} (2015), 2080--2113.

\bibitem{FLP07}
E. Feireisl, Ph. Lauren\c cot and H. Petzeltov\'a,
On convergence to equilibria for the Keller--Segel chemotaxis model,
J. Different. Equ., \textbf{236} (2007), 551--569.

\bibitem{PRL12} X. Fu, L.H. Huang, C. Liu, J.D. Huang, T. Hwa and P. Lenz,
Stripe formation in bacterial systems with density-suppressed motility,
Phys. Rev. Lett., \textbf{108} (2012), 198102.


\bibitem{FJ19}
K. Fujie and J. Jiang,
Global Existence for a Kinetic Model of Pattern Formation with Density-suppressed Motilities,
submitted.

\bibitem{FS2016}
K. Fujie and T. Senba,
Global existence and boundedness of radial solutions to a two dimensional fully parabolic chemotaxis system with general sensitivity,
Nonlinearity, \textbf{29} (2016), 2417--2450.

\bibitem{fs2018}
K. Fujie and T. Senba,
A sufficient condition of sensitivity functions for boundedness of solutions to a parabolic-parabolic chemotaxis system, 
Nonlinearity \textbf{31} (2018), 1639–1672.

\bibitem{fs5}
K. Fujie and T. Senba,
Blowup of solutions to a two-chemical substances chemotaxis system in the critical dimension,
J.\ Differential Equations, \textbf{266} (2019), 942--976.


\bibitem{GM18} 
T. Ghoul and N. Masmoudi,
Minimal mass blowup solutions for the Patlak-Keller-Segel equation,
Commun. Pure Appl. Math., \textbf{71} (2018), 1957--2015.

\bibitem{maximal}
  M. Hieber, J. Pr\"uss,
 Heat kernels and maximal $L^p$ -$L^q$ estimates for parabolic evolution equations,
 Comm.\ Partial Differential Equations {\bf 22} (1997), 1647--1669.



\bibitem{HW01}
D. Horstmann and G.-F. Wang, Blow-up in a chemotaxis model without symmetry assumptions, 
Euro. J. Appl. Math., \textbf{12} (2001), 159--177.



\bibitem{JZ09}
J. Jiang and Y. Zhang,
On convergence to equilibria for a chemotaxis model with volume-filling effect,
Asympt. Anal., \textbf{65} (2009), 79--102.

\bibitem{J18}
J. Jiang,
Convergence to equilibria of global solutions to a degenerate quasilinear Keller--Segel
system,
Z. Angew. Math. Phys., \textbf{69} (2018):130.

\bibitem{JKW18} H.Y. Jin, Y.J. Kim and Z.A. Wang,
Boundedness, stabilization, and pattern formation driven by density-suppressed motility,
SIAM J. Appl. Math., \text{78} (2018), 1632--1657.



\bibitem{Laurencot}
Ph. Lauren\c cot,
Global bounded and unbounded solutions to a chemotaxis system with indirect signal production,
Discrete Contin. Dynam. Syst. Ser. B, \textbf{24} (2019),  6419--6444.


\bibitem{Sciencs11} C. Liu et al.,
Sequential establishment of stripe patterns in an expanding cell population, 
Science, \textbf{334} (2011), 238.

%
\bibitem{mizoguchi_winkler} N. Mizoguchi, M. Winkler,
 Blowup in the two-dimensional Keller--Segel system,
 Preprint.

\bibitem{Nagai97}T. Nagai, T. Senba and K. Yoshida,
Application of the Trudinger--Moser inequality to a parabolic
system of chemotaxis,
Funkcialaj Ekvacioj, \textbf{40} (1997), 411--433.

%
%



\bibitem{ssAMSA2000}
T. Senba and T. Suzuki,
Some structures of the solution set for a stationary system of chemotaxis,
\rm Adv.\ Math.\ Sci.\ Appl.,\ {\bf 10} (2000), 191--224.

\bibitem{ssMAA2001}
T. Senba and T. Suzuki,
Parabolic system of chemotaxis: blowup in a finite and the infinite time,
\rm  Methods Appl.\ Anal.,\ {\bf 8} (2001), 349--367. 

\bibitem{TaoWincritical} Y.S. Tao and M. Winkler,
Critical mass for infinite-time aggregation in a chemotaxis model with indirect signal production, 
J. Eur. Math. Soc. (JEMS), \textbf{19} (2017), 3641--3678.

\bibitem{TaoWin17} Y.S. Tao and M. Winkler,
Effects of signal-dependent motilities in a Keller--Segel-type reaction-diffusion system,
Math. Mod. Meth. Appl. Sci., \text{27} (2017), 1645--1683.

\bibitem{Temam}
R. Temam,
Infinite-dimensional dynamical systems in Mechanics and Physics,
Applied Mathematical Sciences, \textbf{68}, Springer-
Verlag, New York, 1988.

\bibitem{WW2019}
J. Wang and M. Wang,
Boundedness in the higher-dimensional
Keller-Segel model with signal-dependent
motility and logistic growth, 
J. Math. Phys., \textbf{60} (2019), 011507.

\bibitem{Win10} M. Winkler,
Boundedness in the higher-dimensional parabolic-parabolic chemotaxis system with logistic source, Comm. Partial Differential Equations, \textbf{35} (2010), 1516--1537.

\bibitem{win_JDE}
 M. Winkler: 
 Aggregation vs.~global diffusive behavior in the higher-dimensional Keller-Segel model,
 J.\ Differential Equations {\bf 248} (2010), 2889--2905.

\bibitem{Win13} M. Winkler, Finite-time blow-up in the higher-dimensional parabolic-parabolic Keller--Segel system, J. Math. Pures Appl., \textbf{100} (2013), 748--767.


\bibitem{YK17} C. Yoon and Y.J. Kim,
Global existence and aggregation in a Keller--Segel model with Fokker--Planck diffusion,
Acta Appl. Math., \textbf{149} (2017), 101--123.	

\end{thebibliography}
\end{document}